\pdfoutput=1
\documentclass[a4paper,10pt]{amsart}
\usepackage{
 amsmath,
 amssymb,
 caption,
 etoolbox, 
 fancyhdr,
 fullpage,
 mathrsfs,
 mathtools,
 thmtools,
 wrapfig,
 tikz-cd,
 xspace
}

\usepackage[shortlabels]{enumitem}
\setlist[itemize]{leftmargin=\parindent}
\usepackage[pagebackref, hypertexnames=false]{hyperref}
\usepackage[capitalise, noabbrev]{cleveref}

\newlength{\bibitemsep}
\newlength{\bibparskip}
\let\oldthebibliography\thebibliography
\renewcommand\thebibliography[1]{%
 \oldthebibliography{#1}%
 \setlength{\parskip}{\bibparskip}%
 \setlength{\itemsep}{\bibitemsep}%
}

\usepackage[l2tabu, orthodox]{nag}
\theoremstyle{plain}
\newtheorem{theorem}{Theorem}[subsection]

\newtheorem{lemma}[theorem]{Lemma}
\newtheorem{proposition}[theorem]{Proposition}

\newtheorem{definition}[theorem]{Definition}
\newtheorem{notation}[theorem]{Notation}
\newtheorem{conjecture}[theorem]{Conjecture}
\newtheorem{situation}[theorem]{Situation}

\theoremstyle{remark}
\declaretheorem[name=Remark,sibling=theorem,qed={\lower-0.3ex\hbox{$\diamond$}}]{remark}

\declaretheorem[name=Example,sibling=theorem,qed={\lower-0.3ex\hbox{$\diamond$}}]{example}

\newenvironment{smatrix}{\left( \begin{smallmatrix} } {\end{smallmatrix} \right) }

\newcommand{\stbt}[4]{\begin{smatrix}#1 & #2 \\ #3 & #4\end{smatrix}}

\DeclareMathOperator{\GL}{GL}
\DeclareMathOperator{\SO}{SO}
\DeclareMathOperator{\GSp}{GSp}
\DeclareMathOperator{\GSpin}{GSpin}
\DeclareMathOperator{\Gal}{Gal}
\DeclareMathOperator{\Hom}{Hom}
\DeclareMathOperator{\Frob}{Frob}

\DeclareMathOperator{\Iw}{Iw}
\DeclareMathOperator{\Kl}{Kl}
\DeclareMathOperator{\Gr}{Gr}
\DeclareMathOperator{\Fil}{Fil}
\DeclareMathOperator{\Sieg}{Si}

\DeclareMathOperator{\diag}{diag}

\DeclareMathOperator{\myPr}{Pr}\renewcommand{\Pr}{\myPr}

\newcommand{\f}{\mathrm{f}}

\renewcommand{\AA}{\mathbf{A}}
\newcommand{\CC}{\mathbf{C}}

\newcommand{\II}{\mathbf{I}}
\newcommand{\QQ}{\mathbf{Q}}
\newcommand{\RR}{\mathbf{R}}
\newcommand{\VV}{\mathbf{V}}
\newcommand{\ZZ}{\mathbf{Z}}

\newcommand{\Af}{\AA_{\mathrm{f}}}
\newcommand{\QQbar}{\overline{\QQ}}
\newcommand{\Ql}{\QQ_\ell}
\newcommand{\Qp}{\QQ_p}
\newcommand{\Dcris}{\mathbf{D}_{\mathrm{cris}}}
\newcommand{\DdR}{\mathbf{D}_{\mathrm{dR}}}
\newcommand{\Zp}{\ZZ_p}

\newcommand{\cE}{\mathcal{E}}
\newcommand{\cF}{\mathcal{F}}
\newcommand{\cG}{\mathcal{G}}
\newcommand{\cH}{\mathcal{H}}
\newcommand{\cK}{\mathcal{K}}
\newcommand{\cL}{\mathcal{L}}

\newcommand{\cN}{\mathcal{N}}
\newcommand{\cO}{\mathcal{O}}
\newcommand{\cP}{\mathcal{P}}

\newcommand{\cS}{\mathcal{S}}

\newcommand{\cV}{\mathcal{V}}
\newcommand{\cW}{\mathcal{W}}
\newcommand{\cX}{\mathcal{X}}

\newcommand{\RGt}{\widetilde{R\Gamma}}
\newcommand{\wH}{\widetilde{H}}

\newcommand{\bc}{\mathbf{c}}

\newcommand{\fa}{\mathfrak{a}}
\newcommand{\fb}{\mathfrak{b}}
\newcommand{\fX}{\mathfrak{X}}

\newcommand{\Pif}{\Pi_{\mathrm{f}}}
\newcommand{\Sif}{\Sigma_{\mathrm{f}}}

\newcommand{\new} {\mathrm{new}}
\newcommand{\cris}{\mathrm{cris}}

\newcommand{\cl}  {\mathrm{cl}}
\newcommand{\ch}  {\mathrm{ch}}
\newcommand{\dep} {\mathrm{dep}}

\newcommand{\ord} {\mathrm{ord}}

\newcommand{\sph} {\mathrm{sph}}
\newcommand{\mot} {\mathrm{mot}}

\newcommand{\et}{\text{\textup{\'et}}}

\newcommand{\Bor}{\mathrm{B}}
\newcommand{\vno}{\varnothing}

\newcommand{\ucG}{\underline{\mathcal{G}}}

\numberwithin{equation}{section}

\renewcommand{\le}{\leqslant}

\renewcommand{\ge}{\geqslant}

\author{David Loeffler}
\author{Sarah Livia Zerbes}
\thanks{Supported by the following grants: Royal Society University Research Fellowship ``$L$-functions and Iwasawa theory'' and EPSRC Standard Grant EP/S020977/1 (Loeffler); ERC Consolidator Grant ``Euler systems and the Birch--Swinnerton-Dyer conjecture'' (Zerbes).}
\title{P-adic $L$-functions and diagonal cycles for $\GSp_4 \times \GL_2 \times \GL_2$}

\begin{document}
 \renewcommand{\crefrangeconjunction}{--} 
\begin{abstract}
 We develop a (largely conjectural) theory of $p$-adic L-functions interpolating square roots of central $L$-values for automorphic forms on $\GSp_4 \times \GL_2 \times \GL_2$, and a relation between these $p$-adic $L$-functions and families of Galois cohomology classes interpolating algebraic cycles. Our theory is a generalisation of the theory of ``diagonal cycles'' developed by Darmon and Rotger for the $\GL_2$ triple product.
\end{abstract}

 \maketitle

 \makeatletter
 \patchcmd{\@tocline}
 {\hfil}
 {\leaders\hbox{\,.\,}\hfil}
 {}{}
 \makeatother

 \setcounter{tocdepth}{1}
 \tableofcontents

\section{Introduction}

 In this paper, we develop a (largely conjectural) theory of $p$-adic $L$-functions interpolating square roots of central $L$-values for automorphic forms on $\GSp_4 \times \GL_2 \times \GL_2$, and a relation between these $p$-adic $L$-functions and families of Galois cohomology classes interpolating algebraic cycles. Our theory is inspired by the beautiful theory of ``diagonal cycles'' for the group $\GL_2 \times \GL_2 \times \GL_2$, developed by Darmon and Rotger in \cite{darmonrotger14,darmonrotger16} (building on earlier work of Gross, Kudla and Schoen). We briefly recall the outline of their theory.

 \subsection{Summary of the theory for $\GL_2 \times \GL_2 \times \GL_2$}
 Let $f, g, h$ be cuspidal modular newforms, of weights $k_f, k_g, k_h$ respectively, whose nebentype characters satisfy $\chi_f \chi_g \chi_h = 1$. We say the triple is ``$f$-dominant'' if $k_f \ge k_g + k_h$, and similarly for $g$ and $h$; if none of these holds, we say it is ``balanced''.

 \subsubsection*{Split case} We first suppose that there is no finite prime dividing the levels of all three forms (or, more generally, that the local root numbers $\varepsilon_\ell(f \times g \times h)$ are $+1$ for all finite primes $\ell$); we call this the \emph{split case}. In this case, the global root number of the $L$-function $L(f \times g \times h, s)$ is +1 if one of the forms is dominant, and $-1$ for balanced weights.

 The theory of \cite{darmonrotger14} shows that if one allows $f,g,h$ to vary in $p$-adic Hida families $\cF, \cG, \cH$, then there are three $p$-adic $L$-functions $\cL_p^{(\cF)}(\cF \times \cG \times \cH)$, $\cL_p^{(\cG)}(\cF \times \cG \times \cH)$ and $\cL_p^{(\cH)}(\cF \times \cG \times \cH)$, all of which are $p$-adic meromorphic functions on the 3-dimensional parameter space $\fX = \fX_{\cF} \times \fX_{\cG} \times \fX_{\cH}$. These all interpolate \emph{square roots} of central values of $L(f \times g \times h, s)$, for specialisations $(f,g,h)$ of $(\cF,\cG,\cH)$, but in different regions: $\cL_p^{(\cF)}(\cF \times \cG \times \cH)$ interpolates these values for specialisations with $f$ dominant, and similarly for $\cL_p^{(\cG)}$ and $\cL_p^{(\cH)}$.

 In the balanced region, there is no interesting $L$-value to interpolate, since the global root number is\footnote{Modulo a finite number of ``exceptional'' specialisations, which are those where $(f, g, h)$ are all of weight 2 and trivial character, with $a_p(f) a_p(g) a_p(h) = 1$; this is related to exceptional-zero phenomena at $p$.} $-1$. However, in place of an $L$-value, we have an interesting Galois cohomology class: for balanced triples $(f, g, h)$ one can define a ``diagonal cycle'' $\Delta(f, g, h)$, which is an algebraic cycle on a product of Kuga--Sato varieties, and the image of this cycle under the \'etale cycle class map gives a cohomology class for the 8-dimensional self-dual Galois representation associated to $(f, g, h)$. The results of \cite{darmonrotger16,DRfamilies,BSV1}  show that these classes interpolate to an analytic family of Galois cohomology classes $\Delta(\cF, \cG, \cH)$ over $\fX$.

 So one has four objects varying analytically over $\fX$: one family of cohomology classes, with an interpolating property in the balanced region; and three analytic functions, i.e.~families of scalars, with interpolating properties in the other three regions. The keystone of the theory, the \emph{explicit reciprocity law} for diagonal cycles, shows that these are related by the Perrin-Riou regulator map of $p$-adic Hodge theory: one can project $\Delta(\cF, \cG, \cH)$ to any of three different 1-dimensional subquotients of the Galois representation locally at $p$, and the images of these projections under the regulator maps give the three $L$-functions $\cL_p^{(\cF)}$, $\cL_p^{(\cG)}$ and $\cL_p^{(\cH)}$. This reciprocity law has a wealth of deep arithmetic applications, including cases of the equivariant Birch--Swinnerton-Dyer conjecture for Artin twists of elliptic curves (see \cite{darmonrotger16}).

  \begin{wrapfigure}{l}{0.4\linewidth}
 \tikzset{every picture/.style={line width=0.75pt}} 

 \begin{tikzpicture}[x=0.75pt,y=0.75pt,yscale=-1,xscale=1]

 \draw    (100,230) -- (100,42) ;
 \draw [shift={(100,40)}, rotate = 450] [color={rgb, 255:red, 0; green, 0; blue, 0 }  ][line width=0.75]    (10.93,-3.29) .. controls (6.95,-1.4) and (3.31,-0.3) .. (0,0) .. controls (3.31,0.3) and (6.95,1.4) .. (10.93,3.29)   ;
 \draw    (100,230) -- (288,230) ;
 \draw [shift={(290,230)}, rotate = 180] [color={rgb, 255:red, 0; green, 0; blue, 0 }  ][line width=0.75]    (10.93,-3.29) .. controls (6.95,-1.4) and (3.31,-0.3) .. (0,0) .. controls (3.31,0.3) and (6.95,1.4) .. (10.93,3.29)   ;
 \draw    (100,150) -- (180,230) ;
 \draw    (100,150) -- (210,40) ;
 \draw    (290,120) -- (180,230) ;
 \draw    (122.12,112.12) -- (147.88,137.88) ;
 \draw [shift={(150,140)}, rotate = 225] [fill={rgb, 255:red, 0; green, 0; blue, 0 }  ][line width=0.08]  [draw opacity=0] (10.72,-5.15) -- (0,0) -- (10.72,5.15) -- (7.12,0) -- cycle    ;
 \draw [shift={(120,110)}, rotate = 45] [fill={rgb, 255:red, 0; green, 0; blue, 0 }  ][line width=0.08]  [draw opacity=0] (10.72,-5.15) -- (0,0) -- (10.72,5.15) -- (7.12,0) -- cycle    ;
 \draw    (132.12,197.88) -- (157.88,172.12) ;
 \draw [shift={(160,170)}, rotate = 495] [fill={rgb, 255:red, 0; green, 0; blue, 0 }  ][line width=0.08]  [draw opacity=0] (10.72,-5.15) -- (0,0) -- (10.72,5.15) -- (7.12,0) -- cycle    ;
 \draw [shift={(130,200)}, rotate = 315] [fill={rgb, 255:red, 0; green, 0; blue, 0 }  ][line width=0.08]  [draw opacity=0] (10.72,-5.15) -- (0,0) -- (10.72,5.15) -- (7.12,0) -- cycle    ;
 \draw    (192.12,172.12) -- (227.88,207.88) ;
 \draw [shift={(230,210)}, rotate = 225] [fill={rgb, 255:red, 0; green, 0; blue, 0 }  ][line width=0.08]  [draw opacity=0] (10.72,-5.15) -- (0,0) -- (10.72,5.15) -- (7.12,0) -- cycle    ;
 \draw [shift={(190,170)}, rotate = 45] [fill={rgb, 255:red, 0; green, 0; blue, 0 }  ][line width=0.08]  [draw opacity=0] (10.72,-5.15) -- (0,0) -- (10.72,5.15) -- (7.12,0) -- cycle    ;

 \draw (72,32.4) node [anchor=north west][inner sep=0.75pt]    {$k_{h}$};
 \draw (281,240.4) node [anchor=north west][inner sep=0.75pt]    {$k_{g}$};
 \draw (165,142.4) node [anchor=north west][inner sep=0.75pt]    {$\Delta $};
 \draw (104,82.4) node [anchor=north west][inner sep=0.75pt]    {$\mathcal{L}^{(\mathcal{H})}_{p}$};
 \draw (106,202.4) node [anchor=north west][inner sep=0.75pt]    {$\mathcal{L}^{(\mathcal{F})}_{p}$};
 \draw (236,202.4) node [anchor=north west][inner sep=0.75pt]    {$\mathcal{L}^{(\mathcal{G})}_{p}$};
 \draw (181,122.4) node [anchor=north west][inner sep=0.75pt]    {$( -)$};
 \draw (104,178.4) node [anchor=north west][inner sep=0.75pt]    {$( +)$};
 \draw (106,56.4) node [anchor=north west][inner sep=0.75pt]    {$( +)$};
 \draw (251,178.4) node [anchor=north west][inner sep=0.75pt]    {$( +)$};
 \end{tikzpicture}
 \captionsetup{width=0.9\linewidth}
 \caption{Interpolation regions for $\GL_2 \times \GL_2 \times \GL_2$}
 \label{fig:tripleprod}

 \vspace{-3ex}
 \end{wrapfigure}
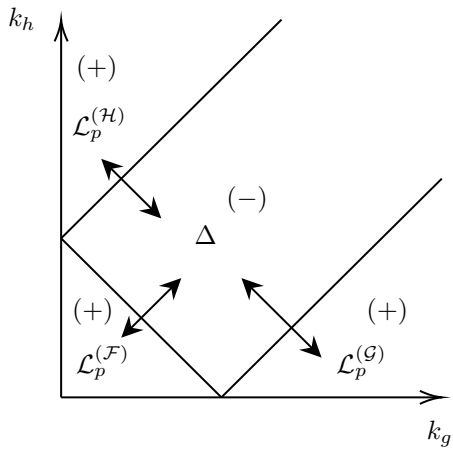

 We summarise this visually in \cref{fig:tripleprod}. Here we have fixed a value of $k_f$ and considered $k_g$ and $k_h$ as variables; the slanted rectangular strip is the balanced region, and the three triangular regions are those where $f$, $g$ or $h$ is dominant. The arrows denote reciprocity laws, relating the $p$-adic $L$-functions given by analytic continuation from a ``sign $+1$'' region to the cohomology classes given by analytic continuation from the neighbouring ``sign $-1$'' region.

 \subsubsection*{Non-split cases} If one relaxes the assumptions on the levels of the Hida families, it can happen that there are some primes $\ell \ne p$ such that $\varepsilon_\ell(f, g, h) = -1$ (for some, and hence all, classical specialisations of $(\cF, \cG,\cH)$). If there is an even number of such primes (the ``indefinite'' case), then the pattern of global root numbers is the same as in \cref{fig:tripleprod}, and the theory is similar, although one needs to work with non-split Shimura curves over $\QQ$ in place of modular curves (ramified precisely at those primes where the $\varepsilon$-factor is $-1$).

 The ``definite'' case, when the root numbers are $-1$ at an odd number of primes $\ell$, presents new phenomena. In this case, the pattern of global root numbers switches around: the global root number is $+1$ for balanced weights and $-1$ for unbalanced ones. So one might hope for a single ``balanced'' $p$-adic $L$-function for $\cF \times \cG \times \cH$, interpolating central values in the balanced region, and three families of Galois cohomology classes, each interpolating algebraic cycles in one of the three dominant regions; and there should be reciprocity laws relating all three families of cohomology classes to the $p$-adic $L$-function.

 This balanced $p$-adic $L$-function has been constructed by Greenberg and Seveso \cite{greenbergseveso20}, using the arithmetic of the definite quaternion algebra ramified at $\infty$ and at the finite primes of root number $-1$. However, we do not know of any construction of algebraic cycles for specialisations in the dominant regions, so the theory is rather less complete in the definite case than the indefinite one.

 \subsubsection*{Eisenstein cases}

  One can also investigate the case when one or more of the cuspidal Hida families $(\cF, \cG, \cH)$ is replaced with a family of Eisenstein series. In this way, one can regard the Beilinson--Flach cohomology classes for $\GL_2 \times \GL_2$ cusp forms \cite{BDR-BeilinsonFlach}, and Beilinson--Kato classes for a single cusp form \cite{kato04}, as degenerate cases of the diagonal-cycle cohomology classes. This forms one of the ``trilogies'' of global cohomology classes outlined in \cite{BCDDPR}.

  In these cases, one obtains an extra structure which is not present in the fully cuspidal case: by varying the defining data of the Eisenstein series, one can obtain global Galois cohomology classes not only over $\QQ$ but over cyclotomic fields $\QQ(\mu_m)$ for varying $m$, satisfying norm-compatibility relations as $m$ changes. These are \emph{Euler systems} in the sense of \cite{rubin00}\footnote{Note that some works such as \cite{BCDDPR} use the term ``Euler system'' in a more general sense, to mean any family of globally-defined cohomology classes. However, in this work we use the term in its stricter sense.}. The existence of an Euler system \emph{sensu stricto} has powerful consequences for Selmer groups and the Bloch--Kato conjecture; see \cite{kato04} and \cite{KLZ17} for examples of such results.

 \subsection{Aims of this paper}

  In the present paper, we shall attempt to develop an analogous theory with $\GL_2 \times \GL_2 \times \GL_2$ replaced by the larger group $\GSp_4 \times \GL_2 \times \GL_2$. This can be seen as a $p$-adic counterpart of the Gross--Prasad conjecture \cite{grossprasad92} for $\GSpin(4) \times \GSpin(5)$, while the triple-product theory is related to the Gross--Prasad conjecture for $\GSpin(3) \times \GSpin(4)$, since the split forms of the spin groups are given by
  \begin{align*}
   \GSpin(3) &\cong \GL_2, &
   \GSpin(4) &\cong \GL_2 \times_{\GL_1} \GL_2, &
   \GSpin(5) \cong \GSp_4.
  \end{align*}
  (In this optic, one can also interpret the $p$-adic $L$-functions and reciprocity laws for Heegner points of \cite{BDP13} as a $p$-adic Gross--Prasad theory for $\GSpin(2) \times \GSpin(3)$, taking a non-split form of $\GSpin(2)$ compact at $\infty$.)

  We shall consider the variation of central $L$-values and Galois cohomology in 2-parameter families of automorphic forms for $\GSp_4 \times \GL_2 \times \GL_2$, given by tensoring a fixed automorphic representation of $\GSp_4$ with Hida families on the two $\GL_2$ factors.

  As we vary over specialisations of such a family, the local root numbers at the finite primes are constant, but the local root numbers at $\infty$ are not: they depend on the weights of the specialisations concerned. So a prominent role in this theory is played by the diagram of interpolation regions shown in \cref{fig:GGP} below, which has nine regions labelled $\{ a, a',b,b',c,d,d',e,f\}$, each labelled with a sign ``$+$'' or ``$-$'' indicating the corresponding Archimedean root number. This diagram is the $\GSp_4 \times \GL_2 \times \GL_2$ counterpart of \cref{fig:tripleprod}.

  If the product of the local root numbers at the finite places is $+1$ (the indefinite case) then in the best of all possible worlds we should have six $p$-adic $L$-functions, one for each of the six regions $\{a, a',c,d,d',f\}$ of sign $+1$, interpolating square roots of central $L$-values for specialisations of our Hida families lying in that region. There should also be three analytic families of cohomology classes, one for each of the sign $-1$ regions $\{b,b',e\}$; and for each pair of neighbouring regions on the diagram, there should be a reciprocity law, relating the cohomology class for a sign $-1$ region to the $p$-adic $L$-function for the neighbouring sign $+1$ region.

  If the local root numbers at the finite places is $-1$ (the definite case), then we should expect instead that there are $p$-adic $L$-functions in regions $\{b,b',e\}$ and families of cohomology classes in the other six regions, again with reciprocity laws for each pair of neighbouring regions.

  Sadly, we cannot carry out this programme in anything like its full form. With the presently available methods, we can do the following:
  \begin{itemize}
   \item \emph{P-adic $L$-functions, indefinite case}: In the indefinite case, we expect $p$-adic $L$-functions for regions $\{a, a', c, d, d',f\}$. We shall give a construction for region $(f)$ below; and region $(c)$ is covered in forthcoming work of Bertolini--Seveso--Venerucci. Regions $(d)$ and $(d')$ may also be accessible, but regions $(a)$ and $(a')$ are completely out of reach.

   \item \emph{Algebraic cycles, indefinite case}: We explain below the construction of a family of cohomology classes for indefinite families interpolating algebraic cycles in region $(e)$. We expect such families also to exist for regions $(b)$ and $(b')$, but we have no idea how to construct these.

   \item \emph{Reciprocity laws, indefinite case}: Since we can construct a family of classes in region $(e)$, we can hope for four reciprocity laws relating it to each of the adjacent regions $\{c,d,d',f\}$. The reciprocity law relating regions $(e)$ and $(f)$ is closely related to the results of \cite{LZ20}; we state below a partial result towards this, whose proof is given in a sequel to the present paper (although this falls short of a full proof). We understand that a reciprocity law relating regions $(e)$ and $(c)$ will be treated in the forthcoming work of Bertolini--Seveso--Venerucci. It seems likely that these methods will also apply to relate region $(e)$ to regions $(d)$ and $(d')$ once the relevant $p$-adic $L$-functions have been constructed, but this is a problem for future work.

   \item \emph{P-adic $L$-functions, definite case}: In the definite case, the tools are available to construct a $p$-adic $L$-function in region $(e)$, although we shall not give full details here. Regions $(b)$, $(b')$ seem to be more difficult.

   \item \emph{Algebraic cycles and reciprocity laws, definite case}: Here we are entirely at a loss. There are 6 regions in which one might hope for an algebraic cycle, but we do not have a candidate construction for any of them. This clearly also rules out any hope of proving reciprocity laws. (This is perhaps not surprising, since we are equally stuck in the analogous case for $\GL_2 \times \GL_2 \times \GL_2$.)
  \end{itemize}

  In this paper, we shall not consider the ``Eisenstein'' cases -- where one or both of the two $\GL_2$ Hida families is a family of Eisenstein series. In such cases, the local root numbers at finite primes are automatically $+1$, so non-split inner forms do not arise. We refer to \cite{LSZ17, LPSZ1, LZ20} for the case of two Eisenstein series, giving an Euler system for cusp forms on $\GSp_4$; and to the very recent preprint \cite{HJS20} of Hsu, Jin and Sakamoto (based on a project at the 2018 Arizona Winter School) for the case where exactly one of the Hida families is taken to be Eisenstein, giving an Euler system for cusp forms on $\GSp_4 \times \GL_2$.

 \subsection*{Acknowledgements}

  The present paper has its origins in a talk given by Rodolfo Venerucci at the conference ``P-adic modular forms and p-adic L-functions'' held in Como, Italy, in 2019, describing his forthcoming work with Massimo Bertolini and Marco Seveso on relating families of algebraic cycles in region $(e)$ to $p$-adic $L$-functions in region $(c)$. The present paper grew out of an attempt to understand the relation between this construction and our works with Pilloni and Skinner. We are very grateful to the organisers of the Como conference for making this encounter possible, and to Rodolfo and Marco for generously sharing their ideas with us during the Como conference and a subsequent visit to Milan.


\section*{Conventions}

 \subsubsection*{Groups}As in \cite[\S 2]{LSZ17}, $G$ denotes the symplectic group $\GSp_4$, given by
 \[ \{ (g, \nu) \in \GL_4 \times \mathbf{G}_m: {}^t\! g J g = \nu J\},
  \qquad J = \begin{smatrix}
        &&&1\\
        &&1\\
        &-1\\
        -1
       \end{smatrix}.\]
 $\Bor_G$ denotes the upper-triangular Borel subgroup of $G$, and $P_{\Sieg}$ and $P_{\Kl}$ denote the Siegel and Klingen parabolic subgroups containing $\Bor_G$.

 Let $H$ denote the group $\GL_2 \times_{\GL_1} \GL_2$. We consider $H$ as a subgroup of $G$ via the embedding
 \[ \iota: \left[\begin{pmatrix} a & b \\ c & d\end{pmatrix}, \begin{pmatrix} a' & b'\\ c'& d'\end{pmatrix}\right]\mapsto\begin{pmatrix} a &&& b\\ & a' & b' & \\ & c' & d' & \\ c &&& d \end{pmatrix}.
 \]
 We also consider $H$ as a subgroup of $\tilde H = \GL_2 \times \GL_2$ in the obvious way. Thus we obtain a diagonal embedding of $H$ into $G \times \tilde{H}$.

 \subsubsection*{Dirichlet characters} If $\chi: (\ZZ / N\ZZ)^\times \to R^\times$ is a finite-order character, for some ring $R$, then we write $\widehat{\chi}$ for the character of $\QQ^\times \backslash \AA^\times / \RR^{\times}_{>0}$ satisfying $\widehat{\chi}(\varpi_\ell) = \chi(\ell)$ for primes $\ell \nmid N$, where $\varpi_\ell$ is a uniformizer at $\ell$. We identify finite-order characters of $\QQ^\times \backslash \AA^\times / \RR^{\times}_{>0}$ with characters of $\Gal(\QQ^{\mathrm{ab}} / \QQ)$ via the Artin map, normalised to send $\varpi_\ell$ to a geometric Frobenius $\Frob_{\ell}^{-1}$.

 (Note that that the restriction of $\widehat{\chi}$ to $\widehat{\ZZ}^\times \subset \Af^\times$ is the inverse of $\chi$, and the Galois character associated to $\widehat{\chi}$ is the composite of $\chi$ with the inverse of the mod $M$ cyclotomic character.)

\section{The classical Gross--Prasad conjecture for \texorpdfstring{$\GSpin(4) \times \GSpin(5)$}{GSpin(4) x GSpin(5)}}

 \subsection{The local conjecture}

  Let $F$ be a number field, $v$ a place of $F$ and $\Pi_v$ an irreducible smooth representation of $G_v \coloneqq G(F_v)$. We let $\Sigma_v = \Sigma_{1, v} \boxtimes \Sigma_{2, v}$ be an irreducible representation of $\tilde{H}_v = \GL_2(F_v) \times \GL_2(F_v)$. We are interested in the space
  \[
   \Hom_{H_v}(\Pi_v \otimes \Sigma_{v}, \CC),
  \]
  where $H_v$ is embedded diagonally in $G_v \times \tilde{H}_v$ as above. This space is trivially zero unless the central characters satisfy $\chi_{\Pi,v}\chi_{\Sigma_{1}, v} \chi_{\Sigma_{2}, v} = 1$, so we shall suppose henceforth that this condition holds.

  \begin{conjecture}[Local Gross--Prasad conjecture]
   \label{conj:locGGP}
   Let $\Phi(\Pi_v)$ be the local $L$-packet containing $\Pi_v$. Then we have
   \[ \sum_{\Pi_v^\sharp \in \Phi(\Pi_v)}\dim \Hom_{H_v}(\Pi_v^{\sharp} \otimes \Sigma_{v}, \CC) = \begin{cases}
       1 & \text{if $\varepsilon_v(\Pi_v \times \Sigma_{v}) = 1$,} \\
       0 & \text{if $\varepsilon_v(\Pi_v \times \Sigma_{v}) = -1$.}
      \end{cases}
   \]
   Moreover, there is an explicit recipe specifying the unique member $\Pi_v^\sharp$ of the $L$-packet for which the Hom-space is non-zero.
  \end{conjecture}

  This conjecture is known if $\chi_{\Pi,v}$ is a square in the group of characters of $F_v^\times$, since in this case we we may reduce to the case where $\chi_{\Pi,v}$ is trivial and apply the main theorem of \cite{waldspurger12b}. The general case will be treated in forthcoming work of Emory and Takeda (Emory, \emph{pers.comm.})

  \begin{remark}\label{rem:innerforms}
   We have only stated part of the conjecture here. Firstly, the restriction of $\Sigma_v$ from $\tilde{H}_v$ to $H_v$ can break up into a sum of multiple irreducible subrepresentations (forming an $L$-packet for $H_v$), and the conjecture predicts which of these $H_v$-summands supports the Gross--Prasad period. Moreover, if the root number is $-1$, then the conjecture also predicts the existence of non-zero local periods for the functorial transfers of $\Pi_v$ and $\Sigma_{v}$ to a specific non-split inner form of $G_v \times \tilde{H}_v$.
  \end{remark}

 \subsection{The global conjecture}

  Let $\Pi, \Sigma$ be cuspidal automorphic representations of $G(\AA_F)$ and of $\tilde{H}(\AA_F)$ such that $\chi_{\Pi} \cdot \chi_{\Sigma_1} \cdot \chi_{\Sigma_2} = 1$, and such that $\Pi_v$ and $\Sigma_{i, v}$ are tempered for all $v$. We suppose that $\Hom_{H_v}(\Pi_v \times \Sigma_{v}, \CC)$ is 1-dimensional for all $v$ (i.e.~that $\varepsilon_v = +1$ and $\Pi^{\sharp}_v = \Pi_v$ in the notation of \cref{conj:locGGP}).

  \begin{conjecture}[Global Gross--Prasad conjecture]
   \label{conj:GGP}
   The $H(\AA_F)$-invariant linear functional on $\Pi \otimes \Sigma$ given by
   \[ \mathcal{P}(\varphi,\sigma) = \int_{[H]} \varphi(\iota(h)) \sigma(h) \, \mathrm{d}h,\quad [H]\coloneqq H(F)Z_G(\AA) \backslash H(\AA) \]
   is non-zero if and only if $\Lambda(\Pi \times \Sigma, \tfrac{1}{2}) \ne 0$.
  \end{conjecture}

  If $\chi_\Pi$ is trivial (so $\Pi$ and $\Sigma$ factor through $\SO_5$ and $\SO_4$ respectively), this is \cite[Conjecture 14.8]{grossprasad92}; and a more precise conjecture, relating $|P(\varphi, \sigma)|^2$ to the product of $\Lambda(\Pi \times \Sigma, \tfrac{1}{2})$ and a collection of local terms depending on $\varphi$ and $\sigma$, is formulated in \cite{ichinoikeda10}. Both conjectures have been generalised to allow arbitrary $\chi_{\Pi}$ by Emory \cite{emory19}. These conjectures are open in general, but have been proved in \cite{gantakeda11} for $\Pi$ of Yoshida, or twisted Yoshida, type.

  We shall assume henceforth that $F = \QQ$, and we shall study the Gross--Prasad periods in the case of \emph{algebraic} automorphic representations $\Pi$, $\Sigma$ (contributing to coherent cohomology of Shimura varieties). Our results are thus \emph{logically independent} of the global Gross--Prasad conjecture; however, the conjecture serves as crucial motivation for studying the periods $\mathcal{P}(\varphi, \sigma)$. 

 \subsection{Interpolation regions at \texorpdfstring{$\infty$}{infty}}
  \label{sect:archGGP}

  Following \cite[\S 12]{grossprasad92}, we now make explicit the predictions of the local Gross--Prasad conjecture at $\infty$ for discrete-series representations (and certain limits of discrete series). Let $\Pi_\infty$ be a unitary irreducible representation of $\GSp_4(\RR)$, and $\Sigma_{1, \infty}$, $\Sigma_{2, \infty}$ unitary irreducible representations of $\GL_2(\RR)$, satisfying the following conditions:

  \begin{enumerate}[(i)]

   \item $\Pi_\infty$ is one of the representations of ``weight $(k_1, k_2)$'' in the sense of \cite{LSZ17}, for some integers $k_1 \ge k_2 \ge 2$. Thus $\Pi_\infty$ is a member of an archimedean $L$-packet $\{\Pi_\infty^{\mathrm{H}}, \Pi_\infty^{\mathrm{W}}\}$, with $\Pi_\infty^\mathrm{W}$ generic, and $\Pi_\infty^\mathrm{H}$ holomorphic discrete series (or limit of discrete series if $k_2 = 2$), corresponding to holomorphic vector-valued Siegel modular forms of weight $(k_1, k_2)$.

   \item $\Sigma_{1,\infty}$ and $\Sigma_{2, \infty}$ are the holomorphic discrete series, or limit of discrete series, representations corresponding to holomorphic modular forms of weights $c_1, c_2$, for some integers $c_i \ge 1$.

   \item The central characters of all three representations are trivial on $\RR_{>0}^\times$, and satisfy $\chi_{\Pi_\infty} \chi_{\Sigma_{1, \infty}} \chi_{\Sigma_{2, \infty}} = 1$ (implying $c_1 + c_2 = k_1 + k_2 \bmod 2$).

  \end{enumerate}
  As above we let $\Sigma_\infty = \Sigma_{1, \infty} \boxtimes \Sigma_{2, \infty}$. As an $H(\RR)$-representation, this is the direct sum of two irreducible representations $\Sigma_\infty^{\mathrm{H}} \oplus \Sigma_\infty^{\mathrm{M}}$, forming an $L$-packet for $H(\RR)$; here $\Sigma_\infty^{\mathrm{H}}$ is generated by modular forms which are either holomorphic in both variables or anti-holomorphic in both, and $\Sigma_\infty^{\mathrm{M}}$ is generated by the forms of ``mixed'' type (holomorphic in one variable and anti-holomorphic in the other).

  The predictions of the Gross--Prasad conjecture in this setting depend on which of the shaded regions in Figure \ref{fig:GGP} contains the point $(c_1, c_2)$. In the figure, the signs ``$+$'' or ``$-$'' indicate the local $\varepsilon$-factor in each region.

  \begin{itemize}

   \item For the regions $(a), (a'), (c), (d), (d'), (f)$, the sign is $+1$, and there is a non-trivial period for some representations $\Pi_\infty^\sharp$ and $\Sigma_\infty^{\sharp}$ in the $L$-packets of $\Pi_\infty$ and $\Sigma_\infty$. These are given as follows:

   \medskip
   \renewcommand{\arraystretch}{1.2}
   \begin{tabular}{ccccccc}
    \hline
    Region & $\Pi^\sharp_\infty$ & $\Sigma^\sharp_\infty$ &\quad& Region & $\Pi^\sharp_\infty$ & $\Sigma^\sharp_\infty$\\
    \hline
    $(a)$,$(a')$ & $\Pi_\infty^\mathrm{W}$ & $\Sigma_\infty^\mathrm{H}$ &&
    $(d)$,$(d')$ & $\Pi_\infty^\mathrm{W}$ &  $\Sigma_\infty^\mathrm{M}$\\
    $(c)$        & $\Pi_\infty^\mathrm{H}$ & $\Sigma_\infty^\mathrm{H}$ &&
    $(f)$        & $\Pi_\infty^\mathrm{W}$ & $\Sigma_\infty^\mathrm{H}$\\
    \hline
   \end{tabular}\medskip

   (The above table is closely related to the branching computations of \cite{harriskudla92}.)

   \item For the regions $(b), (b'), (e)$, the local sign is $-1$, so the GGP periods are zero for all representations in the $L$-packet of $\Pi_\infty \times \Sigma_\infty$.

  \end{itemize}

  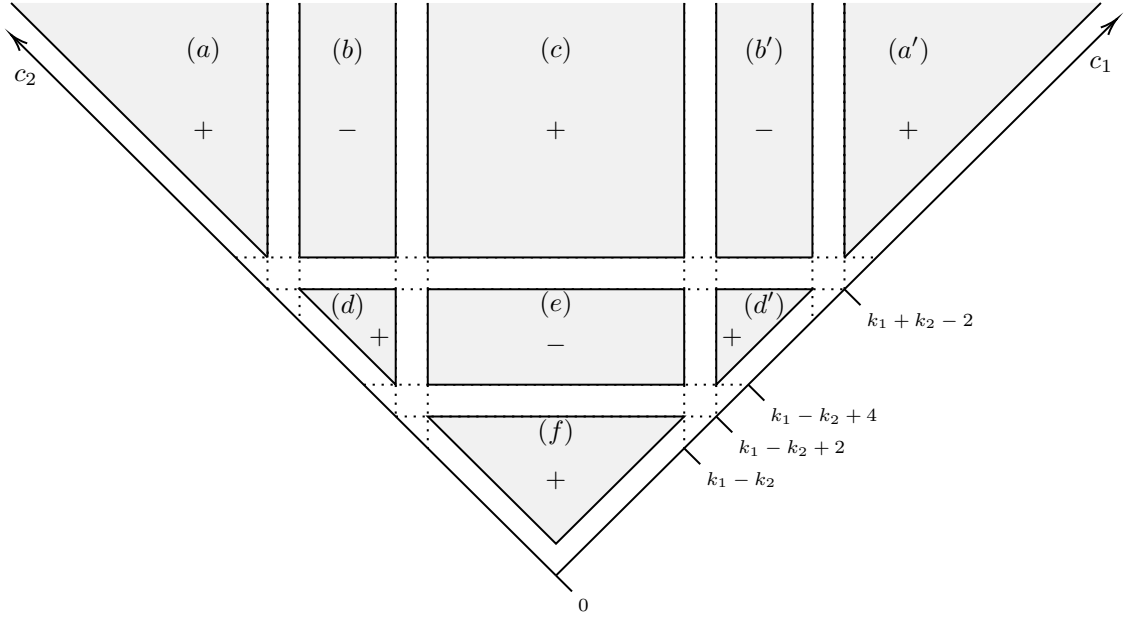
\begin{figure}[t]
   \caption{Regions in $(c_1, c_2)$ plane for archimedean Gross--Prasad}\medskip

   \tikzset{every picture/.style={line width=0.75pt}} 
\begin{tikzpicture}[x=0.6pt,y=0.6pt,yscale=-1,xscale=1]

 \draw    (340,360) -- (1.41,21.41) ;
 \draw [shift={(0,20)}, rotate = 405] [color={rgb, 255:red, 0; green, 0; blue, 0 }  ][line width=0.75]    (10.93,-3.29) .. controls (6.95,-1.4) and (3.31,-0.3) .. (0,0) .. controls (3.31,0.3) and (6.95,1.4) .. (10.93,3.29)   ;
 \draw    (340,360) -- (688.59,11.41) ;
 \draw [shift={(690,10)}, rotate = 495] [color={rgb, 255:red, 0; green, 0; blue, 0 }  ][line width=0.75]    (10.93,-3.29) .. controls (6.95,-1.4) and (3.31,-0.3) .. (0,0) .. controls (3.31,0.3) and (6.95,1.4) .. (10.93,3.29)   ;
 \draw  [dash pattern={on 0.84pt off 2.51pt}]  (240,0) -- (240,260) ;
 \draw  [dash pattern={on 0.84pt off 2.51pt}]  (260,0) -- (260,280) ;
 \draw  [dash pattern={on 0.84pt off 2.51pt}]  (420,0) -- (420,280) ;
 \draw  [dash pattern={on 0.84pt off 2.51pt}]  (440,0) -- (440,260) ;
 \draw  [dash pattern={on 0.84pt off 2.51pt}]  (220,240) -- (460,240) ;
 \draw  [dash pattern={on 0.84pt off 2.51pt}]  (160,180) -- (520,180) ;
 \draw  [dash pattern={on 0.84pt off 2.51pt}]  (140,160) -- (540,160) ;
 \draw  [dash pattern={on 0.84pt off 2.51pt}]  (180,0) -- (180,200) ;
 \draw  [dash pattern={on 0.84pt off 2.51pt}]  (160,0) -- (160,160) ;
 \draw  [dash pattern={on 0.84pt off 2.51pt}]  (500,0) -- (500,200) ;
 \draw  [dash pattern={on 0.84pt off 2.51pt}]  (520,0) -- (520,160) ;
 \draw  [fill={rgb, 255:red, 241; green, 241; blue, 241 }  ,fill opacity=1 ] (420,240) -- (260,240) -- (260,180) -- (420,180) -- cycle ;
 \draw  [draw opacity=0][fill={rgb, 255:red, 241; green, 241; blue, 241 }  ,fill opacity=1 ] (420,160) -- (260,160) -- (260,0) -- (420,0) -- cycle ;
 \draw  [draw opacity=0][fill={rgb, 255:red, 241; green, 241; blue, 241 }  ,fill opacity=1 ] (500,160) -- (440,160) -- (440,0) -- (500,0) -- cycle ;
 \draw  [draw opacity=0][fill={rgb, 255:red, 241; green, 241; blue, 241 }  ,fill opacity=1 ] (240,160) -- (180,160) -- (180,0) -- (240,0) -- cycle ;
 \draw  [draw opacity=0][fill={rgb, 255:red, 241; green, 241; blue, 241 }  ,fill opacity=1 ] (160,0) -- (160,160) -- (0,0) -- cycle ;
 \draw  [fill={rgb, 255:red, 241; green, 241; blue, 241 }  ,fill opacity=1 ] (240,180) -- (240,240) -- (180,180) -- cycle ;
 \draw  [fill={rgb, 255:red, 241; green, 241; blue, 241 }  ,fill opacity=1 ] (500,180) -- (440,240) -- (440,180) -- cycle ;
 \draw  [draw opacity=0][fill={rgb, 255:red, 241; green, 241; blue, 241 }  ,fill opacity=1 ] (680,0) -- (520,160) -- (520,0) -- cycle ;
 \draw  [fill={rgb, 255:red, 241; green, 241; blue, 241 }  ,fill opacity=1 ] (260,260) -- (420,260) -- (340,340) -- cycle ;
 \draw  [dash pattern={on 0.84pt off 2.51pt}]  (240,260) -- (440,260) ;
 \draw    (0,0) -- (160,160) ;
 \draw    (160,0) -- (160,160) ;
 \draw    (180,0) -- (180,160) ;
 \draw    (240,0) -- (240,160) ;
 \draw    (260,0) -- (260,160) ;
 \draw    (420,0) -- (420,160) ;
 \draw    (440,0) -- (440,160) ;
 \draw    (500,0) -- (500,160) ;
 \draw    (520,0) -- (520,160) ;
 \draw    (680,0) -- (520,160) ;
 \draw    (180,160) -- (240,160) ;
 \draw    (260,160) -- (420,160) ;
 \draw    (440,160) -- (500,160) ;
 \draw  [dash pattern={on 0.84pt off 2.51pt}]  (160,0) -- (160,180) ;
 \draw  [dash pattern={on 0.84pt off 2.51pt}]  (520,0) -- (520,180) ;
 \draw    (440,260) -- (450,270) ;
 \draw    (520,180) -- (530,190) ;
 \draw    (340,360) -- (350,370) ;
 \draw    (460,240) -- (470,250) ;
 \draw    (420,280) -- (430,290) ;

 \draw (340,80) node   [align=left] {$\displaystyle +$};
 \draw (120,80) node   [align=left] {$\displaystyle +$};
 \draw (210,80) node   [align=left] {$\displaystyle -$};
 \draw (0,40) node [anchor=north west][inner sep=0.75pt]   [align=left] {$\displaystyle c_{2}$};
 \draw (671,32) node [anchor=north west][inner sep=0.75pt]   [align=left] {$\displaystyle c_{1}$};
 \draw (532,193) node [anchor=north west][inner sep=0.75pt]  [font=\scriptsize] [align=left] {$\displaystyle k_{1} + k_{2} -2$};
 \draw (560,80) node   [align=left] {$\displaystyle +$};
 \draw (470,80) node   [align=left] {$\displaystyle -$};
 \draw (340,215) node   [align=left] {$\displaystyle -$};
 \draw (340,300) node   [align=left] {$\displaystyle +$};
 \draw (230,210) node   [align=left] {$\displaystyle +$};
 \draw (450,210) node   [align=left] {$\displaystyle +$};
 \draw (120,30) node   [align=left] {$\displaystyle ( a)$};
 \draw (560,30) node   [align=left] {$\displaystyle ( a')$};
 \draw (210,30.5) node   [align=left] {$\displaystyle ( b)$};
 \draw (470,30) node   [align=left] {$\displaystyle ( b')$};
 \draw (340,30) node   [align=left] {$\displaystyle ( c)$};
 \draw (340,189.5) node   [align=left] {$\displaystyle ( e)$};
 \draw (210,190.5) node   [align=left] {$\displaystyle ( d)$};
 \draw (470,190.5) node   [align=left] {$\displaystyle ( d')$};
 \draw (340,270) node   [align=left] {$\displaystyle ( f)$};
 \draw (452,273) node [anchor=north west][inner sep=0.75pt]  [font=\scriptsize] [align=left] {$\displaystyle k_{1} - k_{2} +2$};
 \draw (472,253) node [anchor=north west][inner sep=0.75pt]  [font=\scriptsize] [align=left] {$\displaystyle k_{1} - k_{2} +4$};
 \draw (432,293) node [anchor=north west][inner sep=0.75pt]  [font=\scriptsize] [align=left] {$\displaystyle k_{1} - k_{2}$};
 \draw (352,373) node [anchor=north west][inner sep=0.75pt]  [font=\scriptsize] [align=left] {0};

\end{tikzpicture}


%

   \label{fig:GGP}
  \end{figure}

  \begin{remark} \
   \begin{enumerate}[(i)]
    \item Note that regions $(b)$, $(b')$, $(d)$, $(d')$ and $(e)$ are empty in the case $k_2 = 2$ (when $\Pi_\infty$ is a non-regular limit of discrete series).

    \item The full conjecture predicts that in the sign $-1$ case we should obtain non-trivial periods after replacing $\Pi_\infty \times \Sigma_\infty$ with its Langlands transfer to a specific inner form of $(G \times H)_{/\RR}$. This inner form is $\GSpin(4, 1) \times \GSpin(4, 0)$ for regions $(b), (b')$, and $\GSpin(5, 0) \times \GSpin(4, 0)$ for region $(e)$. (Note that $\GSpin(4, 0) \cong D^\times \times_{\GL_1} D^\times$, where $D$ is the non-split quaternion algebra over $\RR$, and $\GSpin(4, 1)$ and $\GSpin(5, 0)$ both arise as unitary groups of rank 2 Hermitian spaces over $D$.)
    \qedhere
   \end{enumerate}
  \end{remark}

 \subsection{Local data at finite places}
  \label{sect:testvec}

  If $\ell$ is a finite prime, and $\Pi_\ell$ is generic, then the new-vector theory of \cite{robertsschmidt07} and \cite{okazaki} defines a canonical 1-dimensional subspace $\Pi_\ell^{\new} \subset \Pi_\ell$. Similarly, the more familiar new-vector theory for $\GL_2$ gives canonical lines in $\Sigma_{1, \ell}$ and $\Sigma_{2, \ell}$, so we obtain a canonical line $\left(\Pi_\ell \otimes \Sigma_{\ell}\right)^{\new}$ in the tensor product $\Pi_\ell \otimes \Sigma_\ell$.

  Unfortunately, if $c_\ell$ is a basis of $\Hom_{H_\ell}(\Pi_\ell \otimes \Sigma_\ell, \CC)$, it is not always true that $c_\ell$ restricts to a non-zero linear functional on this 1-dimensional subspace $(\Pi_\ell \otimes \Sigma_\ell)^{\new}$; and constructing explicit vectors in $\Pi_\ell \otimes \Sigma_\ell$ on which $c_\ell$ is non-vanishing (``test vectors'') is a difficult problem in general. However, it is clear that we can always find an element $\gamma_{\ell} = (\gamma_{0, \ell}, \gamma_{1, \ell}, \gamma_{2, \ell}) \in H_\ell\backslash (G_\ell \times \tilde{H}_\ell)$ such that $c_\ell$ is non-vanishing on the subspace $\gamma_\ell \cdot \left(\Pi_\ell \otimes \Sigma_{\ell}\right)^{\new}$, since the $\CC$-span of these vectors as $\gamma_\ell$ varies is the whole of $\Pi_\ell \otimes \Sigma_\ell$.

  \begin{remark}
   We can, of course, take $\gamma_\ell = 1$ for all but finitely many primes $\ell$, since the spherical vector is a test vector if $\Pi_\ell$ and $\Sigma_{\ell}$ are unramified.
  \end{remark}

  In our constructions, we shall interpolate Gross--Prasad periods in which the local vectors away from $p$ lie in $\gamma_\ell \cdot (\Pi_\ell \otimes \Sigma_\ell)^{\new}$, with $\Pi$ and $\Sigma$ varying in $p$-adic families, and $\gamma_\ell$ is some arbitrary but fixed element of $G_\ell \times \tilde{H}_\ell$ for each $\ell$. (The choice of local data at $p$ and at $\infty$ will require much closer attention, of course).

\section{Tensor products of Hida families}

 Here we shall briefly recall the theory of $p$-adic families of ordinary eigenforms for $\GL_2$ (Hida families), and explain how to form $p$-adic families of self-dual representations of $\GSp_4 \times \GL_2 \times \GL_2$ from these. These self-dual families will be the setting for our main conjectures and results in the later sections.

 \subsection{Automorphic representations of $\GL_2$}

  We briefly recall the dictionary between ``classical'' and ``adelic'' descriptions of modular forms (since there are multiple conventions in use, and we want to specify our conventions unambiguously).

  Let $\cH$ be the upper half-plane. We define the \emph{weight $k$ action} of $\GL_2^+(\QQ)$ on functions $f: \cH\to \CC$ by $(f \mid_k \stbt{a}{b}{c}{d})(\tau) = (ad-bc) (c\tau + d)^{-k} f(\tfrac{a\tau+b}{c\tau+ d})$. By an \emph{adelic modular form} of weight $k$ and level $U$, where $U$ is an open compact in $\GL_2(\Af)$, we mean a function $f: \GL_2(\Af) \times \cH \to \CC$ which satisfies $f(\gamma g u, -) = f(g, -) \mid_k \gamma^{-1}$ for all $\gamma \in \GL_2^+(\QQ)$ and $u \in U$, and such that $f(g, -)$ is bounded at $\infty$ for all $g$.

  Let $g = \sum_{n \ge 1} a_n(g) q^n$  be a normalised, cuspidal, new modular eigenform of weight $k$, level $N$ and character $\chi$; this is meant in the classical sense, so $g$ is a holomorphic function on the upper half-plane $\cH$, invariant under the weight $k$ action of $\Gamma_1(N)$, and transforming under the diamond operator $\langle d \rangle$ by $\chi(d)$, where $\langle d \rangle$ denotes any element of $\operatorname{SL}_2(\ZZ)$ congruent to $\stbt{\star}{\star}0 d\bmod N$.

  We can identify $g$ uniquely with an adelic modular form of level $K_1(N) \coloneqq \{\stbt{\star}{\star}{0}{1} \bmod N\}$ whose restriction to $\{1\} \times \cH$ is the original $g$; with some trepidation we denote this adelic form also by $g$. Let $\Sif'(g)$ denote the space of adelic modular forms spanned by the $\GL_2(\Af)$-translates of $g$. Then:

  \begin{itemize}
  \item There exists a (non-unitary) cuspidal automorphic representation $\Sigma'(g)$, whose lowest $K_\infty$-type subspace is $\Sif'(g)$.

  \item The central character of $\Sigma'(g)$ is the character $|\cdot|^{(2-k)/2} \widehat{\chi}$ of $\AA^\times / \QQ^\times$, where $|\cdot|$ is the ad\`ele norm.

  \item For $\ell \nmid N$, the Hecke operator $T_\ell$, given by the double coset $\GL_2(\ZZ_\ell) \stbt{\varpi_\ell}{}{}{1} \GL_2(\ZZ_\ell)$, acts on the $\GL_2(\ZZ_\ell)$-invariants of $\Sigma'_\ell(g)$ as $a_\ell(g)$.

  \item The form $g$ is the new vector of $\Sif'(g)$ (the unique vector stable under $K_1(N)$).
  \end{itemize}

  We let $\Sigma(g) = \Sigma'(g) \otimes |\cdot|^{(k-2)/2}$, which is unitary; note, however, that $\Sif'(g)$ is always definable over the field $\QQ(\{ a_\ell(g): \ell \nmid N\})$, which is a finite extension of $\QQ$, while $\Sif(g)$ is not definable over a number field if $k$ is odd.

 \subsection{Hida families}
  \label{sect:hida}

  We fix (for the remainder of this paper) a prime $p > 2$, and an isomorphism $\QQbar_p \cong \CC$. We let $\cO$ be the ring of integers of some finite extension $L /\Qp$. Let $\Lambda$ be the Iwasawa algebra $\cO[[\Zp^\times]]$. We shall consider a $\Lambda$-algebra $\II$ which is a normal integral domain, and is finite and projective as a $\Lambda$-module. Thus $\II$ is a complete local ring. Extending $L$ if necessary, we may suppose $L$ is the maximal algebraic extension of $\Qp$ contained in $\operatorname{Frac} \II$.

  An \emph{arithmetic point} is a homomorphism of $\cO$-algebras $\phi: \II \to \bar{L}$ whose restriction to group-like elements $\Zp^\times \subset \Lambda^\times$ is given by $x \mapsto x^{k_\phi} \chi_{\phi}(x)$, for some $k_{\phi} \ge 1$ and finite-order character $\chi_{\phi}$; we denote this simply by $k_\phi + \chi_{\phi}$.

  \begin{definition}
   Let $N \ge 1$ be coprime to $p$. A \emph{cuspidal Hida family} of tame level $N$, with coefficients in $\II$, is a formal power series $\cG = \sum a_n q^n \in \II[[q]]$, with $a_1 = 1$ and $a_p \in \II^\times$, such that for all arithmetic points $\phi$ such that $k_\phi \ge 2$, the specialisation $\phi(\cG)$ is the $q$-expansion of a normalised Hecke eigenform $g_{\phi} \in S_{k_{\phi}}(\Gamma_1(Np^c))$ for some $c \ge 1$, which is new at $N$, and on which the diamond operators at $p$ act via $\chi_{\phi}$.
  \end{definition}

  \begin{remark}
   The term \emph{Hida family} is used here in a slightly different sense than in \cite{KLZ17} for example; the objects we call ``families'' here were called ``branches'' in \emph{op.cit.}, so they correspond to minimal prime ideals (not to maximal ideals) of the ordinary Hecke algebra.
  \end{remark}

  Note that we do permit, but do not require, the specialisations at weight 1 arithmetic points to be classical modular forms. We let $\fX^{\cl}(\cG)$ denote the set of arithmetic points $\phi$ such that $g_{\phi}$ is classical (hence including all $\phi$ with $k_{\phi} \ge 2$, and possibly also some $\phi$ with $k_{\phi} = 1$). We say $\phi \in \fX^{\cl}(\cG)$ is \emph{crystalline} if $g_{\phi}$ is the ordinary $p$-stabilisation of a newform of level $N$, and we let $\fX^{\cris}(\cG)$ be the set of crystalline points. Note that a crystalline point must have $\chi_{\phi}$ trivial, and this is also a sufficient condition if $k_{\phi} \ne 2$. (When $k_\phi = 2$ this fails, since some specialisations with $k_\phi = 2$ and $\chi_{\phi}$ trivial may have local components at $p$ which are unramified twists of the Steinberg representation.)

  Any Hida family $\cG$ has a well-defined tame character $\chi_{\cG}$, which is a Dirichlet character of conductor dividing $N$. It also has a well-defined ``weight-character'' at $p$, which is a residue class $\bar{c}$ modulo $p-1$, corresponding to the component of $\operatorname{Spec} \Lambda$ over which $\II$ lies. Note that $(-1)^{\bar{c}} = \chi_{\cG}(-1)$.

  \begin{theorem}[Hida, Wiles]
   Any classical $p$-ordinary eigenform is the specialisation of some Hida family. If the eigenform has weight $\ge 2$ then this family is unique.
  \end{theorem}

  \begin{notation}
   For $\cG$ a Hida family over $\II$ and $\phi \in\fX^{\cl}(\cG)$, we write $\Sigma(\cG, \phi) = \bigotimes'_v \Sigma_\ell(\cG, \phi)$ for the unitary automorphic representation of $\GL_2(\AA)$ corresponding to $g_{\phi}$, and $\Sigma'(\cG, \phi)$ for its arithmetic twist.
  \end{notation}

 \subsection{Self-dual twists}

  Let $\cG_1$, $\cG_2$ be two Hida families (of possibly different tame levels $N_1, N_2$ and coefficient rings $\II_1, \II_2$). Let $\II = \II_1 \mathop{\hat\otimes}_{\cO} \II_2$.

  Since the central characters at $p$ vary over specialisations of $\cG_1$ and $\cG_2$, we will need to make one extra choice\footnote{This construction is somewhat modelled on the ``square root'' constructions appearing in the theory of big Heegner points \cite{howard07}.}. If the tame characters satisfy $\chi_{\cG_1}(-1) \chi_{\cG_2}(-1) = 1$, then we let $\bar{c}$ be one of the two residue classes modulo $p-1$ with $\bar{c}_1 + \bar{c}_2 = 2\bar{c}$. This choice determines a character $\bc: \Zp^\times \to \II^\times$ with $\bc_1 + \bc_2 = 2\bc$, where $\bc_i$ are the universal characters. Similarly, if $\chi_{\cG_1}(-1) \chi_{\cG_2}(-1) = -1$, we choose $\bc$ such that $\bc_1 + \bc_2 = 2\bc + 1$.

  In either case, we define a \emph{classical point} of $\II$ to be a ring homomorphism $\phi = \phi_1 \otimes \phi_2: \II\to \QQbar_p$ such that the $\phi_i$ are classical points and, in addition, $\phi \circ \bc$ is a locally-algebraic character. If $\phi_i$ has weight-character $c_i + \chi_i$, then this amounts to assuming that $c_1 + c_2$ has the same parity as $\bar{c}_1 + \bar{c}_2$ and $\chi_1(-1)\chi_2(-1) = 1$. We then have $\phi \circ \bc = \lfloor \frac{c_1 + c_2}{2}\rfloor + \tau_{\phi}$ for some finite-order character $\tau_{\phi}$ which is a square root of $\chi_1\chi_2$.  We say that $\phi$ is \emph{crystalline} if $\chi_1$, $\chi_2$ and $\tau_{\phi}$ are all trivial; note that this forces $c_1 + c_2$ to lie in a specific congruence class modulo $2p-2$.

  \begin{notation}
   \label{not:tripleG}
   We write $\ucG$ for the data of a triple $(\cG_1, \cG_2, \bc)$ as above, and $\fX^{\cl}(\ucG)$, $\fX^{\mathrm{cris}}(\ucG)$ the sets of its classical and crystalline points. For $\phi \in \fX^{\cl}(\ucG)$ we write
   \begin{align*}
    \Sigma_1(\ucG, \phi) &\coloneqq \Sigma(\cG_1, \phi_1),&
    \Sigma_2(\ucG, \phi) &\coloneqq \Sigma(\cG_2, \phi_2) \otimes \widehat{\tau}_\phi^{-1},&
    \Sigma(\ucG, \phi) &\coloneqq  \Sigma_1(\ucG, \phi) \boxtimes \Sigma_2(\ucG, \phi).
   \end{align*}
   We write simply $\Sigma(\phi)$ if $\ucG$ is clear from context.
  \end{notation}

  Here $\hat\tau_{\phi}$ is the adelic character associated to the Dirichlet character $\tau_{\phi}$, as before. Note that the central character of $\Sigma(\phi)$ is $\widehat{\chi}_{\ucG} = \widehat{\chi}_{\cG_1} \widehat{\chi}_{\cG_2}$; in particular, it is independent of $\phi$, and is unramified at $p$ (this is the purpose of the twist $\widehat{\tau}_{\phi}^{-1}$).

  \begin{notation}
   We say $\phi \in \fX^{\cl}$ is \emph{fully ramified} if the finite-order characters $\tau_{\phi}$ and $\chi_1/\tau_{\phi}$ of $\ZZ_p^\times$ are both non-trivial, so the local Rankin--Selberg $L$-factor $L(\Sigma_p(\phi), s)$ is identically 1; and we let $\fX^{\mathrm{ram}}$ be the set of fully ramified points.
  \end{notation}

 \subsection{Constancy of root numbers} We consider a pair $(\Pi, \ucG)$ as follows:

  \begin{situation}
   \label{sit:pair}
   Let $\Pi$ be a cuspidal automorphic representation of $\GSp_4$, with $\Pi_\infty$ of weight $(k_1, k_2)$ for some $k_1 \ge k_2 \ge 2$, as in \S\ref{sect:archGGP}; and let $\ucG = (\cG_1, \cG_2, \bc)$ be a triple as in \cref{not:tripleG}, with $\chi_{\Pi} \cdot \chi_{\cG_1} \cdot \chi_{\cG_2} = 1$. We suppose, in addition, that $\Pi$ is not a CAP representation, and that $\Pi_p$ is unramified.
  \end{situation}

  This will be the setting for all the theorems and conjectures in the remainder of this paper. Since $\Pi$ is not CAP, it is either of type A (general type) or B (Yoshida type) in Arthur's classification \cite{arthur04,geetaibi18}; in particular, $\Pi_v$ is tempered for every $v$.

  \begin{lemma}
   Let $\ell$ be a finite prime. Then the local root number $\varepsilon_\ell(\Pi_\ell \times \Sigma_\ell(\ucG, \phi))$ for $\phi \in \fX^{\cl}(\ucG)$ is independent of the specialisation $\phi$, and we write it as $\varepsilon_\ell(\Pi \times \ucG)$. If $\ell$ does not divide $\gcd(N_1, N_2)$, or if $\Pi_\ell$ is unramified (including $\ell = p$), then we have $\varepsilon_\ell(\Pi \times \ucG) = 1$.
  \end{lemma}

  \begin{proof}
   The case of $\ell = p$ is easily disposed of: in this case the local root numbers are automatically $+1$, since we are assuming $\Pi_p$ to be unramified. So we may suppose $\ell \ne p$. Recall that generic irreducible smooth representations of $\GL_2(\QQ_\ell)$ can be divided into three ``species'' (terminology from \cite{loefflerweinstein11}): either principal series, special (twist of Steinberg), or supercuspidal. It follows from results of Hida and Dimitrov \cite[\S 6.2]{dimitrov14} that for $\ell \ne p$, the species of $\Sigma_{i, \ell}(\phi_i)$ is independent of the arithmetic specialisation $\phi_i$.

   If either of the families $\cG_i$ is principal-series at $\ell$, or if $\Pi_\ell$ is unramified, then all of the root numbers are clearly $+1$. So we may assume that $\ell \mid (N_1, N_2)$ and both of the $\cG_i$ are supercuspidal or special at $\ell$. It follows from the results quoted above that the local representations are given by families of twists of a fixed representation: we can find representations $\pi_{i, \ell}$ with coefficients in $\QQbar_p$, and elements $\alpha_{i, \ell} \in \II_i^\times$, such that $\Sigma_{i, \ell}(\phi_i) = \pi_{i, \ell} \otimes \phi_i(\alpha_{i, \ell})$ for all arithmetic specialisations $\phi_i$. Hence the representation $\Sigma_{\ell}(\phi)$ is actually independent of $\phi$, since the twist $\tau_{\phi}$ cancels out the twists $\alpha_{i, \ell}$. So the root number is certainly constant over the family.
  \end{proof}

  \begin{definition}
   \label{not:defindef}
   We say the pair $(\Pi, \ucG)$ is
   \begin{itemize}
   \item \emph{split} if $\varepsilon_\ell(\Pi \times \ucG) = 1$ for all finite primes $\ell \ne p$;
   \item \emph{indefinite} if $\# \{ \ell: \varepsilon_\ell(\Pi \times \ucG) = -1\}$ is even (including the split case);
   \item \emph{definite} if $\# \{ \ell: \varepsilon_\ell(\Pi \times \ucG) = -1\}$ is odd.
   \end{itemize}
  \end{definition}

\section{Conjectures I: $p$-adic $L$-functions for $+1$ regions}

 \subsection{A general  conjecture for sign $+1$ regions}

  Let $(\Pi, \ucG)$ be as in \cref{sit:pair}. We suppose in addition that $(\Pi, \ucG)$ is split, and that $\Pi_p$ is \emph{ordinary}, i.e.~its Hecke parameters $(\alpha, \beta, \gamma, \delta)$ at $p$ have valuations $(0, k_2 - 2, k_1 - 1, k_1 + k_2 - 3)$ respectively.

  \begin{notation}
   We decompose the set $\fX^{\cl}(\ucG)$ of classical specialisations of $\ucG$ as
   \[ \fX^{\cl} = \fX_{a} \sqcup \fX_{a'} \sqcup \dots \sqcup \fX_{f}, \]
   according to the regions in \cref{fig:GGP}. For $\spadesuit \in \{a, \dots, f\}$, let $\varepsilon_\infty(\spadesuit)$ denote the corresponding Archimedean root number (i.e.~$-1$ if $\spadesuit \in \{ b, b', e\}$ and $+1$ for the other six regions).
  \end{notation}

  \begin{conjecture}
   \label{conj:pGGP}
   Suppose $\spadesuit$ is a region such that $\varepsilon_\infty(\spadesuit) = +1$, so that all specialisations $\Pi \times \Sigma(\phi)$ with $\phi \in \fX_\spadesuit$ have global root number $+1$. Then there should exist an element
   \[ \cL_p(\Pi \times \ucG, \spadesuit) \in \operatorname{Frac}(\II)\]
   whose values at $\phi \in \fX_{\spadesuit}$ satisfy
   \[ \cL_p(\Pi \times \ucG, \spadesuit)(\phi) = \mathcal{E}_p^{\spadesuit}(\Pi_p \times \Sigma_p(\phi)) \cdot \frac{\cP(\varphi \times \sigma(\phi))}{\Omega_\infty^{\spadesuit}(\Pi \times \Sigma(\phi))}\]
   where $\mathcal{E}_p^{\spadesuit}(\Pi_p \times \Sigma_p(\phi))$ is an Euler factor at $p$, and $\Omega_\infty^{\spadesuit}(\dots)$ an Archimedean period, and $\varphi$ and $\sigma(\phi)$ are appropriate vectors in $\Pi \times \Sigma(\phi)$, whose components at places $v \notin \{p, \infty\}$ are the translates of the new vectors by a fixed element $\gamma_v \in H_v \backslash (G_v \times \tilde{H}_v)$.
  \end{conjecture}

  \begin{remark} \
   \begin{enumerate}
    \item Of course, if the global Gross--Prasad conjecture \ref{conj:GGP} holds, the right-hand side of the conjecture can be more simply expressed as a multiple of the square root of $\Lambda(\Pi \times \Sigma(\phi), \tfrac{1}{2})$; but the $p$-adic conjecture as formulated above is logically independent of Conjecture \ref{conj:GGP}.
    \item It is important to note here that $\cL_p(\Pi \times \ucG, \spadesuit)$ will depend crucially on the choice of $\spadesuit$. So, having chosen $\Pi \times \ucG$, the conjecture predicts the existence of six different $p$-adic $L$-functions (one for each of the six regions $\{a,a',c,d,d',f\}$ of sign $+1$). These objects all have different interpolating properties (and will involve different periods and Euler factors).

    \item The dependence on $\gamma_\ell$ is an irritant. It seems natural to expect that for each prime $\ell \ne p$, there should be a rank 1 projective $\II[1/p]$-module interpolating (in some appropriate sense) the spaces $\Hom_{H_\ell}(\Pi_\ell \times \Sigma_\ell(\phi), \QQbar_p)$ for varying $\phi \in \fX^{\cl}$. In this case, the dependence of the $p$-adic $L$-function on the choice of $\gamma_\ell$ could be made explicit: it would be given by evaluating a basis of this Hom-space at the $\gamma_\ell$-translates of the new vectors.

    If the local factors of $\cG_1$ and $\cG_2$ are both non-principal-series at $\ell$, then $\Sigma_\ell(\phi)$ is actually independent of $\phi$, and the existence of such an interpolating module is obvious. However, the case where one of the $\cG_i$ is principal-series at $\ell$ is more difficult, and developing a theory of ``$\II$-adic branching laws'' applicable to this case is beyond the scope of the present paper.

    \item For $\ell = p$, if $\phi \in \fX^{\cris}$, we simply choose $\varphi_p$ and $\sigma_p(\phi)$ to be spherical vectors. Other local vectors at $p$ will play a role in the construction, but these do not appear in the final formula -- the relation between these alternate vectors at $p$ and the spherical vectors will be measured by the Euler factor $\cE_p$.

    The question of which vectors to choose at $\infty$, and at $p$ when $\phi$ is not crystalline, is rather more delicate; we do not know what the ``correct'' choices should be for all choices of $\spadesuit$. The case $\spadesuit = (f)$ is explored in more detail in \S\ref{sect:pLf} below.
    \qedhere
   \end{enumerate}
  \end{remark}

 \subsection{Transfer to inner forms}

  We briefly sketch how the conjecture above should be modified if the local signs at some finite places are $-1$. As before, we assume that $(\Pi, \ucG)$ be as in \cref{sit:pair}, with $\Pi$ ordinary at $p$.

  If $(\Pi, \ucG)$ is indefinite (in the sense of \cref{not:defindef}), the full form of \cref{conj:GGP} (see \cref{rem:innerforms}) singles out inner forms $G^{\sharp}, H^{\sharp}$ of $G$ and $H$, with $H^{\sharp} \subset G^{\sharp}$, which are both split at $\infty$; and we can expect to obtain $p$-adic $L$-functions for each of the six $+1$ regions $\{a,a',c,d,d',f\}$, as in the split case, interpolating the Gross--Prasad periods for the transfers of $\Pi$ and $\Sigma_i(\phi)$ to $G^\sharp \times H^\sharp$ for $\phi$ in those regions.

  In the definite case, it is instead the regions $\spadesuit \in \{b, b', e\}$ in which the specialisations of $\Pi \times \ucG$ have global root number +1; and we can expect the existence of three $p$-adic $L$-functions attached to $\Pi \times \ucG$, one for each of these regions. In this case, to obtain non-trivial Gross--Prasad periods we will have to use inner forms $G^{\sharp}, H^{\sharp}$ which are non-split at $\infty$, with Archimedean part $\GSpin(5, 0) \times \GSpin(4, 0)$ in region $(e)$, and $\GSpin(4, 1) \times \GSpin(4, 0)$ in regions $\{ b, b'\}$.

  \subsection{The Euler factor}

   The expected form of the Euler factor $\mathcal{E}_p^{\spadesuit}(\Pi_p \times \Sigma_p(\phi))$ can be given in terms of Galois representations. We let $\VV(\phi)$ denote the unique (up to isomorphism) semisimple $\Gal(\QQbar/\QQ)$-representation such that
   \[ \det(1 - \ell^{-s} \Frob_\ell^{-1}\ |\ \VV(\phi)) = L(\Pi_\ell \times \Sigma_{1, \ell} \times \Sigma_{2, \ell}, s + \tfrac12)\]
   for all but finitely many $\ell$, where $\Frob_\ell$ is an arithmetic Frobenius. Note that $\VV(\phi)\cong \VV(\phi)^*(1)$. As we shall recall in more detail in \cref{sect:conjs2} below, the ordinarity of $\Pi$ and the $\Sigma_i$ implies that there is a unique subrepresentation $\VV^{+}(\phi)$ stable under $G_{\Qp}$ such that $\VV^{+}(\phi)$ has all Hodge--Tate weights $\ge 1$, and the quotient $\VV^{-}(\phi) = \VV(\phi)/\VV^{+}(\phi)$ has all weights $\le 0$. We then set
   \begin{align*}
    \cE_p(\Pi \times \Sigma(\phi)) \coloneqq& \det\left(1 - p^{-1} \varphi^{-1}: \Dcris(\Qp, \VV^{+}(\phi)) \right) \\
    =& \det\left(1 - \varphi: \Dcris(\Qp, \VV^{-}(\phi)) \right).
   \end{align*}
   Note that if $\phi$ is fully ramified, then $\Dcris(\Qp, \VV(\phi)) = \{0\}$, and hence $\cE_p^{\spadesuit}(\Pi_p \times \Sigma_p(\phi)) = 1$. On the other hand, if $\phi$ is crystalline, then $\Dcris(\Qp, \VV(\phi))$ is 16-dimensional, with $\varphi$-eigenvalues $\{ \frac{\alpha \fa_1 \fa_2}{p^{w+1}}, \dots, \frac{\delta \fb_1 \fb_2}{p^{w+1}}\}$, where $(\alpha, \beta,\gamma, \delta)$ are the Hecke parameters of $\Pi_p$,  $(\fa_i, \fb_i)$ are the Hecke parameters of $\Sigma_{i, p}$, and $w = \tfrac{k_1 + k_2 + c_1 + c_2 - 6}{2}$. Moreover, the eigenvalues which appear in the subrepresentation $\VV^{+}$ are precisely those whose valuation is $\le -1$. So we have
   \[ \cE_p^{\spadesuit}(\Pi_p \times \Sigma_p(\phi)) = \prod\left\{ \left(1 - \tfrac{p^w}{\xi}\right): \xi \in \{\alpha \fa_1 \fa_2, \dots, \delta \fb_1 \fb_2\}, v_p(\xi) \le w \right\}.\]
   Since all the possible $\xi$ have complex absolute value $p^{w + \tfrac{1}{2}}$, we have in particular $\cE_p^{\spadesuit}(\Pi_p \times \Sigma_p(\phi)) \ne 0$ for all crystalline $\phi$.

   We can determine exactly which $\xi$'s appear from the regions in \cref{fig:GGP}: we can order the parameters such that $v_p(\fa_i) = 0, v_p(\fb_i) = c_i - 1$, and similarly $(\alpha, \beta, \gamma, \delta)$ have valuations $(0, k_2 - 2, k_1 - 1, k_1 + k_2 - 3)$. Hence the contributing $\xi$'s are those given in the following table. (We have omitted regions $(a'),(b'),(d')$ since these can be filled in by symmetry; and since the $\xi$'s come in pairs with product $p^{w+1}$, and exactly one from each pair contributes, we can omit the remaining columns. See below for the significance of the last column.)

   \begin{table}[ht]
   \caption{Frobenius eigenvalues contributing to $\cE_p^{\spadesuit}$ for crystalline $\phi$}
   \setlength{\tabcolsep}{3 pt}
   \begin{tabular}{c|cccc cccc cc}
   \hline
   $\spadesuit$ &
   $\alpha\fa_1\fa_2$ &
   $\alpha\fa_1\fb_2$ &
   $\alpha\fb_1\fa_2$ &
   $\alpha\fb_1\fb_2$ &
   $\beta\fa_1\fa_2$ &
   $\beta\fa_1\fb_2$ &
   $\beta\fb_1\fa_2$ &
   $\beta\fb_1\fb_2$ & \dots & $P^{\spadesuit}$ \\
   \hline
   (a) & $\checkmark$ & $\times$ & $\checkmark$ & $\times$ & $\checkmark$ & $\times$ & $\checkmark$ & $\times$ && $(\vno, \vno, \Bor)$\\
   (b)& $\checkmark$ & $\checkmark$ & $\checkmark$ & $\times$ & $\checkmark$ & $\times$ & $\checkmark$ & $\times$ && $(\Sieg,\Bor, \Bor)$\\
   (c)& $\checkmark$ & $\checkmark$ & $\checkmark$ & $\times$ & $\checkmark$ & $\checkmark$ & $\checkmark$ & $\times$ && $(\Kl,  \Bor, \Bor)$\\
   (d)& $\checkmark$ & $\checkmark$ & $\checkmark$ & $\checkmark$ & $\checkmark$ & $\times$ & $\checkmark$ & $\times$ && $(\Sieg, \vno, \Bor)$\\
   (e)& $\checkmark$ & $\checkmark$ & $\checkmark$ & $\checkmark$ & $\checkmark$ & $\checkmark$ & $\checkmark$ & $\times$ && $(\Bor, \Bor, \Bor)$\\
   (f) & $\checkmark$ & $\checkmark$ & $\checkmark$ & $\checkmark$ & $\checkmark$ & $\checkmark$ & $\checkmark$ & $\checkmark$ &&  $(\Kl, \vno, \vno)$\\
   \hline
   \end{tabular}
   \end{table}

   \begin{remark}
    \label{rem:parabolics}
    We have placed ourselves in a setting where the $\Sigma_i$ are both nearly ordinary at $p$ (i.e.~ordinary up to twisting), and $\Pi_p$ is ordinary for the Borel subgroup. However, this can be weakened; for instance, for the construction of a $p$-adic $L$-function in region $(f)$ in in \S \ref{sect:pLf} below, we do not need to assume that $\Pi_p$ be Borel ordinary, only that it be Klingen-ordinary.

    The weakest ordinarity condition that is needed to make sense of the conjectures is that the subrepresentation $\VV^+(\phi)$ should exist for all $\phi \in \fX^{\spadesuit}$. This corresponds to be requiring near-ordinarity of $\Pi_p \times \Sigma_p$ for the parabolic subgroup $P^{\spadesuit}$ of $\GSp_4 \times \GL_2 \times \GL_2$ indicated in the right-most column of the above table. Here $B$ denotes the Borel of either $\GSp_4$ or $\GL_2$, and $\varnothing$ denotes the whole group (for which all automorphic representations are vacuously nearly-ordinary).

    This suggests that the $p$-adic $L$-functions $\cL_p(\Pi \times \Sigma, \spadesuit)$ should extend to larger parameter spaces (``big eigenvarieties''), classifying automorphic representations having nearly-ordinary refinements for $P^{\spadesuit}$. See \cite{loeffler20} for further speculation along these lines.
   \end{remark}

 \subsection{Accessible cases}

  For certain choices of $\spadesuit$, the above conjecture is accessible.

  \subsubsection*{Split cases}
   Let us first assume that the local signs at the finite places are $+1$.

   \begin{itemize}
   \item We shall describe in \S \ref{sect:pLf} below how to construct $\mathcal{L}_{p}(\Pi \times \ucG, \spadesuit)$ for $\spadesuit = (f)$, using the higher Hida theory methods of \cite{LPSZ1}.

   \item Regions $\spadesuit = (d)$ and $(d')$ may be accessible using a variant of the methods of \cite{LPSZ1}; we hope to investigate this case in forthcoming work.

   \item The case $\spadesuit = (c)$ will be treated in forthcoming work of Bertolini--Seveso--Venerucci.
  \end{itemize}

  \subsubsection*{Non-split indefinite cases}

  It may be possible to extend the above results to the setting when the local sign is $+1$ at a (nonzero) even number of finite primes. To carry out such a construction for $\spadesuit = (f)$ (and presumably also $(d)$, $(d')$) would require generalising the higher Hida theory used in \cite{LPSZ1} to certain Hodge-type, but not PEL-type, Shimura varieties for inner forms of $\GSp_4$ which are split locally at $\infty$ (and at $p$) but not globally split.

  \subsubsection*{Definite cases} In a somewhat different direction, if there is an odd number of finite places where the local sign is $-1$, and we take $\spadesuit = (e)$, then the inner forms $G^{\sharp}$ and $H^{\sharp}$ are compact at $\infty$; so the relevant automorphic forms are functions on finite sets. In this case, it should be possible to construct a $p$-adic $L$-function by interpolating branching laws for algebraic representations, generalising the treatment of the totally definite $\GL_2 \times \GL_2 \times \GL_2$ case developed in \cite{greenbergseveso20}.
  \bigskip

  The remaining regions of global sign $+1$, i.e.~regions $\{a, a'\}$ in the indefinite case and $\{b,b'\}$ in the definite case, are much more mysterious. In these cases, we do not know how to interpret the periods $\cP(\varphi \otimes \sigma)$ in terms of any natural cohomology theory, which seems to be a pre-requisite for any kind of $p$-adic interpolation.

\section{Constructions I: the \texorpdfstring{$p$-adic $L$-function in region $(f)$}{p-adic L-function in region (f)}}

 \label{sect:pLf}

 \subsection{Setting}

  Let $(\Pi, \ucG)$ be as in \cref{sit:pair}, and suppose the following additional conditions hold:
  \begin{itemize}
   \item $\Pi$ is Klingen-ordinary at $p$.
   \item $\Pi$ is globally generic.
   \item $(\Pi, \ucG)$ is split in the sense of \cref{not:defindef}.
   \item $k_2 \ge 4$ (a technical condition inherited from \cite{LPSZ1}).
  \end{itemize}

  In this setting, for $\phi \in \fX^{(f)}$, we shall express the Gross--Prasad periods $\cP(\varphi \times \sigma_{\phi})$ for suitable test vectors $(\varphi, \sigma_{\phi})$ as cup-products in coherent cohomology, with $\varphi$ corresponding to a class in $H^2$ and $\sigma_{\phi}$ to a class in $H^0$. We shall then interpolate these cup products $p$-adically as $\phi$ varies, using the techniques of \emph{higher Hida theory} developed in \cite{pilloni20, LPSZ1} to study variation of coherent cohomology in families.

 \subsection{Hida families as $\II$-adic modular forms}

  We would like to interpret $\cG_i$ as $\II$-adic modular forms, in the sense of Katz (as elements of the structure sheaf of the Igusa tower over $X_1(N) \times \II$). However, at this point an annoying technical complication arises. With our conventions for modular curves (inherited from \cite{LSZ17} and ultimately from \cite{kato04}), a modular form of level $N$ which is defined over a ring $R$ does not necessarily have $q$-expansion coefficients in $R$, only in $R[\mu_N]$. Also, the comparison between classical and $p$-adic modular forms is only $\GL_2(\Af^p)$-equivariant up to a twist when the character at $p$ is nontrivial; see \cite[\S 4.1--4.4]{LPSZ1}. Correcting appropriately for this, we obtain the following result:

  \begin{proposition}
   Let $\cG$ be a Hida family over $\II$ of tame level $N$. Then there exists an $\II[\mu_N]$-adic cusp form, which we denote also by $\cG$, whose level away from $p$ is $K_1(N)$ and whose $q$-expansion is the power series $\cG \in \II[[q]]$. The specialisation of this form at any $\phi \in \fX^{\cl}(\cG)$ with $\chi_{\phi} = 1$ is the unique normalised eigenvector $g_{\phi}^{\ord} \in \Sigma(\cG, \phi)$ which is stable under $K_1(N) \cap \{ \stbt{\star}{\star}{}{\star} \bmod p\}$ and lies in the ordinary eigenspace for $U_p$.

   If $\phi \in \fX^{\cl}$ has weight-character $k_{\phi} + \chi_{\phi}$ with $\chi_{\phi}$ nontrivial, then the specialisation of $\cG$ at $\phi$ is the adelic modular form $g_\phi^{\new} \otimes \widehat{\chi}_{\phi}^{-1}$, which transforms under $\stbt{x}{\star}{0}{\star} \in \GL_2(\Zp)$ via $(\widehat{\chi}_{\phi})_p(x^{-1}) = \chi_\phi(x)$.
  \end{proposition}

  Confusingly, $g_\phi^{\new} \otimes \widehat{\chi}_{\phi}^{-1}$ has the same $q$-expansion at $\infty$ as $g_\phi^{\new}$, but the two are not identical as adelic modular forms; $g_\phi^{\new}$ is an element of $\Sif'(g_\phi)$ invariant under $\stbt{\star}{\star}{}{1} \subset \GL_2(\Zp)$, while $g_\phi^{\new} \otimes \widehat{\chi}_{\phi}^{-1}$ is an element of $\Sif'(g_{\phi} \otimes \chi_{\phi}^{-1})$ invariant under $\stbt1\star{}\star$. (However, this twist will not trouble us unduly, since we shall only consider cases where the twists on $\cG_1$ and $\cG_2$ are opposite.)

 \subsection{Local data at $p$}

  We briefly recall some of the local zeta-integral computations carried out in \cite[\S 8]{LPSZ1} and \cite[\S 20]{LZ20}.

  \subsubsection{Local vectors for $\Pi$}

   Let $\{\alpha, \beta, \gamma, \delta\}$ be the Hecke parameters of $\Pi_p$, normalised such that all four are algebraic integers of complex absolute value $p^{(k_1 + k_2 - 3)/2}$, and ordered such that $\alpha\delta = \beta\gamma = p^{(k_1 + k_2 - 3)} \chi_{\Pi}(p)$ and $0 \le v_p(\alpha) \le \dots \le v_p(\delta)$.

   The Klingen-ordinarity of $\Pi_p$ implies that $\frac{\alpha\beta}{p^{k_2 - 2}}$ is a $p$-adic unit, and it is the unique unit eigenvalue of the integrally-normalised Hecke operator $U_{2, \Kl} = p^{k_1 - 3} \left[ \Kl(p) \diag(p^2, p, p, 1) \Kl(p)\right]$ on the $\Kl(p)$-invariants of $\Pi_p$ for any $n \ge 1$, where
   \[ \Kl(p) = \{ g \in G(\Zp): g \bmod p \in P_{\Kl} \}\]
   is the level $p$ Klingen parahoric.

   Let $\cW(\Pi_p)$ denote the local Whittaker model, and $w^{\sph} \in \cW(\Pi_p)$ the unique spherical vector with $w^{\sph}(1) = 1$. We let $w_{\alpha\beta}^{\Kl}$ denote the unique vector in $\cW(\Pi_p)^{\Kl(p)}$ which lies in the $U_{2, \Kl} = \frac{\alpha\beta}{p^{k_2 - 2}}$ eigenspace and satisfies $w_{\alpha\beta}^{\Kl}(1) = 1$.

   \begin{proposition}
     Let $u_{\Kl}$ be any element of $G(\Zp)$ whose first column is \emph{not} of the form $(0, \star, \star, 0)$ or $(\star, 0, 0, \star)$ modulo $p$. Then, for any smooth $H_p$-representation $\Sigma_{p}$, any $H_p$-equivariant homomorphism $\mathfrak{z}_p: \cW(\Pi_p) \otimes \Sigma_p \to \CC$, and any $\sigma \in \Sigma_p$, the sequence
     \[
      \left(\tfrac{p^{k_1 + k_2 - 1}}{\alpha\beta}\right)^{n}
      \mathfrak{z}_p\left(u_{\Kl} \cdot
      J\cdot
      \diag(p^{2n}, p^n ,p^n, 1)\cdot
      w_{\alpha\beta}^{\Kl}, \sigma \right)
     \]
    is independent of $n$ for $n \gg 0$.
   \end{proposition}

   \begin{proof}
    This follows from the fact that $u_{\Kl}$ represents the open $H$-orbit on the Klingen flag variety $G / P_{\Kl}$; it is a simple instance of the theory developed in \cite{loeffler-spherical}. For sufficiently large $n$, the vector $\sigma \in \Sigma_p$ will be preserved by the principal congruence subgroup mod $p^n$ in $H(\Zp)$. If this $n$ is also $\ge 1$, then the open-orbit condition implies that we can find vectors in this subgroup whose conjugates by $u_{\Kl}\cdot J$ give coset representatives for the $U_{2, \Kl}$-operator acting on $w_{\alpha\beta}^{\Kl}$.
   \end{proof}

   \begin{notation}
    We write $\mathfrak{z}_p\left(u_{\Kl} \cdot w_{\alpha\beta}^{\Kl, \infty}, \sigma \right)$ for the limit of the above sequence.
   \end{notation}

   \begin{remark}
    Here we are interpreting $w_{\alpha\beta}^{\Kl, \infty} = \left(\left(\tfrac{p^{k_1 + k_2 - 1}}{\alpha\beta}\right)^{n} J\cdot \diag(p^{2n}, p^n ,p^n, 1)\cdot w_{\alpha\beta}^{\Kl} \right)_{n \ge 1}$ as a collection of eigenvectors for the transpose operators $U_{2,\Kl}'$ at level
    $\begin{smatrix}
     \star&\star&\star&\star\\
     p^n&\star&\star&\star\\
     p^n&\star&\star&\star\\
     p^{2n}&p^n&p^n&\star\end{smatrix} \subset G(\Zp)$,
    compatible under the normalised trace maps; cf.~\cite[Proposition 5.7]{LPSZ1}. One can check that the projection of this element to level $\Kl(p^n)$ coincides with the vector denoted $w_n^{\Kl, \prime}$ in \cite[\S 20.2.3]{LZ20}.
   \end{remark}

  \subsubsection{Local vectors for $\Sigma$: unramified case}

   Let $\phi \in \fX^{\cl}(\ucG)$ and write $\Sigma = \Sigma(\ucG, \phi)$. As above, we write this as $\Sigma_1 \boxtimes \Sigma_2$ where $\Sigma_1 = \Sigma(\cG_1, \phi_1)$ and $\Sigma_2 = \Sigma(\cG_2, \phi_2) \otimes \tau_\phi^{-1}$. Both $\Sigma_{1, p}$ and $\Sigma_{2, p}$ are generic representations of $\GL_2(\Qp)$.

   If $\phi \in \fX^{\cris}$, then both $\Sigma_{1, p}$ and $\Sigma_{2, p}$ are unramified (and $\tau_{\phi}$ is trivial); and we let $w_1^{\sph}$ and $w_2^{\sph}$ be the corresponding spherical Whittaker functions, normalised by $w_i^{\sph}(1) = 1$. We let $\fa_i,\fb_i$ be the Hecke parameters of $\Sigma_{i, p}$, with $|\fa_i| = |\fb_i| = p^{(c_i - 1)/2}$. We also consider the ``$p$-depleted'' Whittaker vectors, which are the unique vectors $w_i^{\mathrm{dep}} \in \cW(\Sigma_{i, p})$ such that $w_i^{\mathrm{dep}}\left(\stbt{x}{}{}{1}\right)= \ch_{\Zp^\times}(x)$.

   \begin{proposition} \
    \label{prop:region-f-zeta}
    \begin{enumerate}
     \item There is a unique basis vector $\mathfrak{z}_p$ of $\Hom_{H(\Qp)}(\cW(\Pi_p) \times \cW(\Sigma_p), \CC)$ such that
     \[ \mathfrak{z}_p(w^{\sph}, w_1^{\sph}, w_2^{\sph}) = 1. \]
     \item Let $u_{\Kl}$ be any element\footnote{Such an element was denoted simply by $\gamma$ in \emph{op.cit.}, but this conflicts with the simultaneous use of $\gamma$ for a Hecke parameter.} of $G(\Zp)$ with first column $(1, 1, 0, 0)$. For $\mathfrak{z}_p$ as in (1), we have
     \[
      \mathfrak{z}_p\left(u_{\Kl} \cdot w_{\alpha\beta}^{\Kl, \infty}, w_1^{\dep}, w_2^{\dep}\right) = \tfrac{q^3}{(q+1)^2(q-1)} \cE^{(f)}(\Pi_p \times \Sigma_p(\phi)).
     \]
    \end{enumerate}
   \end{proposition}

   \begin{proof}
    This is worked out in detail in \cite{loeffler-zeta2}, expanding upon the special cases treated in \cite[\S 8]{LPSZ1}.
   \end{proof}

   Note that
   \begin{multline*}
    \cE^{(f)}(\Pi_p \times \Sigma_p(\phi)) = \left(1 - \tfrac{p^w}{\alpha \fa_1 \fa_2}\right)
       \left(1 - \tfrac{p^w}{\alpha \fb_1 \fa_2}\right)
       \left(1 - \tfrac{p^w}{\alpha \fa_1 \fb_2}\right)
       \left(1 - \tfrac{p^w}{\alpha \fb_1 \fb_2}\right) \\
       \times
       \left(1 - \tfrac{p^w}{\beta \fa_1 \fa_2}\right)
       \left(1 - \tfrac{p^w}{\beta \fb_1 \fa_2}\right)
       \left(1 - \tfrac{p^w}{\beta \fa_1 \fb_2}\right)
       \left(1 - \tfrac{p^w}{\beta \fb_1 \fb_2}\right),
   \end{multline*}
   where $w = \tfrac{k_1 + k_2 + c_1 + c_2 - 6}{2}$.

  \subsubsection{Local vectors for $\Sigma$: ramified case}

   If $\phi$ is non-crystalline, then the definitions of the vectors $w_i^{\mathrm{dep}} \in \cW(\Sigma_{i, p})$ still make sense (but note that if $\tau_{\phi}$ is non-trivial, $\Sigma_{2, p}$ is a ramified twist of $\Sigma_p(\cG_2, \phi_2)$, so $w_2^{\mathrm{dep}}$ is \textit{not} the image of the $p$-depleted Whittaker vector of $\Sigma_p(\cG_2, \phi_2)$).

   \begin{proposition}
    If $\phi$ is fully ramified, then $u_{\Kl} \cdot w^{\Kl, \infty}_{\alpha\beta} \otimes w_1^{\dep} \otimes w_2^{\dep}$ is a test vector for $\Pi_p \otimes \Sigma_p$; that is, if $\mathfrak{z}_p \in \Hom(\cW(\Pi) \times \cW(\Sigma), \CC)$ is non-zero, then we have
    \[
     \mathfrak{z}_p\left(u_{\Kl} \cdot w^{\Kl, \infty}_{\alpha\beta} \otimes w_1^{\dep} \otimes w_2^{\dep}\right) \ne 0.
    \]
   \end{proposition}

   \begin{proof}
    This follows from the twisted zeta integral computations of \cite[\S 20.4]{LZ20}.
   \end{proof}

   \begin{remark}
    In this ramified case there does not seem to be an obvious ``standard'' normalisation of $\mathfrak{z}_p$.
   \end{remark}
 \subsection{Choosing the global data}

  The global Whittaker transform, given by integrating automorphic forms over the compact quotient $N(\QQ) \backslash N(\AA)$ where $N$ is the upper-triangular unipotent subgroup of $\GSp_4$, gives a canonical isomorphism
  \[ \Pi \cong \cW(\Pi) \cong \sideset{}{'}{\bigotimes}_v \cW(\Pi_v).\]
  For all finite places $v$, the space $\cW(\Pi_v)$ has a normalised new-vector $w_v^{\mathrm{new}}$. So, given a vector $w_p \in \cW(\Pi_p)$, we can consider the global Whittaker function
  \[ w_\infty \cdot w_p \cdot \prod_{v \notin \{ p, \infty\}} \gamma_v  w_v^{\mathrm{new}}, \]
  where $\gamma_v$ is an arbitrary element of $\GSp_4(\QQ_v)$ which is the identity if $v$ is unramified, as in \S \ref{sect:testvec}, and $w_\infty$ is the standard Whittaker function at $\infty$, as in \cite[\S 10.2]{LPSZ1}. We denote the images of $w^{\sph}$ and $u_{\Kl} \cdot w^{\Kl, \prime}_n$ by  $\varphi^{\sph}$ and $\varphi^{\Kl}_n$ respectively, suppressing the role of $\gamma$ from the notation.

  The theory for the $\Sigma_i$ is analogous, with the standard Whittaker function being the complex exponential $\stbt y x 0 1 \mapsto y^{c_i/2} \exp(- 2\pi y + 2\pi \sqrt{-1} x )$, so the resulting vectors are holomorphic modular forms in the lowest $K_\infty$-type subspace of $\Sigma_i$. In general these cannot pair non-trivially with $w_\infty$, since the $K_\infty$-types do not match; so we shall use the Maass--Shimura differential operator $\delta$ to raise the weights. That is, we shall consider the periods
  \[ \cP\left(\varphi^{\Kl}_n \times \sigma_1^{[p]} \times \delta^t(\sigma_2^{[p]})\right) \]
  where $t = \frac{k_1 - k_2 -c_1 - c_2 + 2}{2}$ for $\phi$ of weight $(c_1, c_2)$, and $\sigma_i^{[p]}$ has local component $w_i^{\dep}$ at $p$. These are independent of $n$ for $n \gg 0$, and we denote the limiting value simply by $\cP\left(\varphi^{\Kl}_\infty \times \sigma_1^{[p]} \times \delta^t(\sigma_2^{[p]})\right)$. Note that the assumption that $(c_1, c_2)$ lies in region $(f)$ is precisely the inequality needed to ensure that $t \ge 0$. It follows from the computations of \cite{moriyama04} and \cite{lemma17} that the above vector is a local test vector at $\infty$; see \cite[\S 8.5]{LPSZ1}.

  \begin{lemma}
   Suppose $\phi$ is either crystalline, or fully ramified, and that the local conjecture \ref{conj:locGGP} holds at all places.

   If $\cP(\dots)$ is not identically zero as a linear functional on $\Pi \times \Sigma$, then it is non-zero on the vector $\varphi^{\Kl}_n \times \sigma_1^{[p]} \times \delta^t(\sigma_2^{[p]})$, for some choice of the auxiliary $\gamma_v$ at places $v \ne p,\infty$. Moreover, if $\phi$ is crystalline then
   \[
    \cP\left(\varphi^{\Kl}_\infty \times \sigma_1^{[p]} \times \delta^t(\sigma_2^{[p]})\right) = \tfrac{q^3}{(q+1)^2(q-1)}\cE_p^{(f)}(\Pi_p \times \Sigma_p)\cP\left(\varphi^{\sph} \times \sigma_1^{\sph} \times \delta^t(\sigma_2^{\sph})\right),
   \]
   and $\cE_p^{(f)}(\Pi_p \times \Sigma_p) \ne 0$.
  \end{lemma}

  \begin{proof} Clear from the zeta-integral computations above. \end{proof}

 \subsection{The construction}

  We make one final choice: a $\QQbar$-basis $\nu$ of the new subspace of $H^2(\Pi)$, where $H^2(\Pi)$ is the copy of $\Pif$ appearing in the degree 2 coherent cohomology of the Siegel Shimura variety. As in \cite[\S 6.8]{LPSZ1}, comparing $\nu$ with the standard Whittaker function defines a period $\Omega_\infty^W(\Pi, \nu) \in \CC^\times$.

  \begin{theorem}
   There is an element
   \[ \cL_{p, \nu, \gamma}(\Pi \times \ucG; (f)) \in \II \otimes_{\cO} \overline{L}\]
   such that for all $\phi \in \fX_{(f)}(\ucG)$ we have
   \[ \cL_{p, \nu, \gamma}(\Pi \times \ucG; (f)) (\phi) = \frac{\cP\left(\varphi^{\Kl}_\infty \times \sigma_1^{[p]} \times \delta^t(\sigma_2^{[p]})\right)}{\Omega^W_{\infty}(\Pi, \nu)}.  \]

   In particular, if $\phi \in \fX^{(f)}(\ucG)$ is either crystalline, or fully ramified, and the local conjecture \ref{conj:locGGP} holds for $\Pi \times \Sigma(\phi)$ at all places, then we have $\cL_{p, \nu, \gamma}^{(f)}(\Pi \times \ucG) (\phi) \ne 0$ for some $\gamma$ if and only if the Gross--Prasad period $\cP$ is non-zero on $\Pi \otimes \Sigma(\phi)$.
  \end{theorem}

  Note that no periods from the $\Sigma_i$ appear; this is a consequence of the $q$-expansion principle for $\GL_2$, which shows that modular forms with Fourier coefficients in $\QQbar$ are algebraic as coherent cohomology classes.

  \begin{proof}
   After much notational unravelling this follows from Theorem 5.11 and Corollary 6.17 of \cite{LPSZ1}. We take the $R$ of \emph{op.cit.} to be our $\II$, and we take the family of $p$-adic modular forms $\mathcal{E}$ of \emph{op.cit.} to be the pullback to $H$ of the family on $\GL_2 \times \GL_2$ given by
   \[ \cE = \gamma_1 \cdot \cG_1^{[p]} \boxtimes \gamma_2 \cdot \theta^{r-\bc}\cG_2^{[p]}, \]
   where $()^{[p]}$ denotes the $p$-depletion operator, and $r = \lfloor \frac{k_1 - k_2+2}{2}\rfloor$. Thus we have $\cE \in \cS_{\tau_1, \tau_2}(\II)$ for two characters $\tau_1, \tau_2 : \Zp^\times \to \II^\times$ such that $\tau_1 + \tau_2$ is the integer $k_1 - k_2 + 2$.

   If we specialise at a classical point $\phi$, then the specialisation of $\cE$ is the image under the $p$-adic unit root splitting of the nearly-holomorphic form $\sigma_1^{[p]} \times \delta^t(\sigma_2^{[p]})$ (or more precisely of its restriction to $H$). We choose the Klingen-level eigenform $\eta$ of \emph{op.cit.} such that $u_{\Kl} \cdot \eta = \frac{1}{\Omega_\infty^W(\Pi, \nu)} \varphi_{\Kl}$; by construction this lies in the coherent cohomology over $\QQbar$. This gives the stated formula.
  \end{proof}

\section{Interlude: Automorphic Galois representations}
 We now recall some standard facts about Galois representations attached to automorphic representations, as preliminaries for the conjectures and constructions to follow.

 \subsection{Galois representations for $\GL(2)$}

  If $g \in S_k(\Gamma_1(N))$ is any newform with coefficients in $L$, with $k \ge 2$ (and any $N \ge 1$, not necessarily coprime to $p$), we define a 2-dimensional $L$-linear $\Gal(\QQbar/\QQ)$-representation $V_p(g)$ as in \cite{deligne69}, as a direct summand of the \'etale cohomology of $Y_1(N)$ with coefficients in the appropriate \'etale sheaf depending on $k$.

  \begin{remark}
   Up to isomorphism, $V_p(g)$ is characterised by $\operatorname{tr}(\Frob_\ell^{-1} : V_p(g)) = a_\ell(g)$ for $\ell \nmid Np$, where $\Frob_\ell$ is the arithmetic Frobenius. However, the conjectures below will depend upon the specific realisation of $V_p(g)$ in \'etale cohomology, not just its isomorphism class. Note also that the Hodge numbers\footnote{Negatives of Hodge--Tate weights; so the Hodge number of the cyclotomic character is $-1$.} of $V_p(g)$ are $0$ and $k-1$.
  \end{remark}

  We can extend the definition of $V_p(g)$ to any normalised eigenform (not necessarily new); in particular, it applies to eigenforms of level $Np$ that are $p$-stablisations of newforms of level $N$, with $p \nmid N$. In this case, if $g^{\new}$ is the newform corresponding to $g$, then $V_p(g)$ and $V_p(g^{\new})$ are abstractly isomorphic, but they do not coincide as submodules of the \'etale cohomology of $Y_1(Np)$; an explicit isomorphism between them is given by the map $(\Pr^{\alpha})^*$ described in \cite[\S 5.7]{KLZ17}.

  \begin{theorem}[Hida, Ohta]
   Let $\cG$ be a Hida family which is residually non-Eisenstein and $p$-distinguished\footnote{That is, if $\bar{\rho}_{\cG}$ denotes the common mod $p$ Galois representation of all specialisations of $\cG$, then $\bar{\rho}_{\cG}$ is irreducible, and the two characters appearing in the semisimplification of $\bar{\rho}_{\cG} |_{G_{\Qp}}$ are distinct. This is automatic if the weight-character of $\cG$ is not $1 \pmod{p-1}$.}. Then there exists a free rank two $\II$-module $V_p(\cG)$, with a continuous $\II$-linear action of $\Gal(\QQbar / \QQ)$, such that for any $\phi \in \fX^{\cl}(\cG)$ with $k_\phi \ge 2$, we have $V_p(\cG) \otimes_{\II, \phi} \bar{L} = V_p(g_{\phi})$.
  \end{theorem}

  For our constructions it is more convenient to work with the dual $V_p(\cG)^*$. This has a canonical realisation as a quotient of \'etale cohomology. More specifically, if we define
  \[ H^1_{\ord}(N, \Lambda) = e'_{\ord} \varprojlim_r H^1\left(Y_1(Np^r)_{\QQbar}, \cO(1)\right), \]
  where $e'_{\ord}$ is the ordinary projector associated to the dual Hecke operator $U_p'$, then the action of the diamond operators gives $H^1_{\ord}(N)$ the structure of a finite-rank projective $\Lambda$-module, and we can define $V_p(\cG)^*$ as the maximal quotient of $H^1_{\ord}(N, \Lambda) \otimes_{\Lambda} \II$ on which the dual Hecke operators $T_\ell'$ act via $a_\ell(\cG)$. The $p$-distinction hypothesis implies that the localisation of $H^1_{\ord}(N, \Lambda)$ at the maximal ideal corresponding to $\bar\rho_{\cG}$ is free of rank 2 over the Hecke algebra, so $V_p(\cG)^*$ is free of rank 2 over $\II$. It satisfies
  \[ V_p(\cG)^* \cong \left(V_p(\cG) \otimes \widehat{\chi}_{\cG}^{-1}\right)(\bc - 1)\]
  where $\widehat{\chi}_{\cG}$ is a finite-order $\cO^\times$-valued character and $\bc -1$ denotes the composite of the cyclotomic character $\Gal(\QQbar/\QQ) \to \Zp^\times$ with the character $\Zp^\times \to \II^\times$, $x \mapsto x^{\bc - 1}$.

  Finally, we note that if $g$ is ordinary, then $V_p(g)|_{\Gal(\QQbar_p/\Qp)}$ has an unramified subrepresentation $\Fil^1 V_p(g)$, on which $\Frob_p^{-1}$ acts as $a_p(g)$. These subrepresentations interpolate over Hida families:

  \begin{theorem}[Ohta]
   The representation $V_p(\cG)$ has a free rank one $\II$-submodule $\Fil^1 V_p(\cG)$, stable under $\Gal(\QQbar_p / \Qp)$, interpolating the submodules $\Fil^1 V_p(g_\phi)$ for $\phi \in \fX^{\cl}(\cG)$.
  \end{theorem}

 \subsection{Self-dual twists}

  Let us suppose we are given a triple $\ucG = (\cG_1, \cG_2, \bc)$ as in \cref{not:tripleG} with both of the $\cG_i$ being $p$-distinguished, and $\II = \II_1 \mathop{\hat\otimes_{\cO}} \II_2$.
  \newcommand{\htimes}{\mathop{\hat\otimes}}
  Then the representation
  \[ V_p(\ucG) \coloneqq \left(V_p(\cG_1) \htimes_{\cO} V_p(\cG_2)\right) (\mathbf{c}-1) \]
  is free of rank 4 over $\II$, and satisfies
  \[
   V_p(\ucG)^*\otimes\widehat{\chi}_{\cG_1}\widehat{\chi}_{\cG_2} \cong
   \begin{cases}
    V_p(\ucG) & \text{if $\chi_{\cG_1}(-1) \chi_{\cG_2}(-1) = +1$,}\\
    V_p(\ucG)(1) &\text{if $\chi_{\cG_1}(-1) \chi_{\cG_2}(-1) = -1$.}
   \end{cases}
  \]
  In particular, $V_p(\ucG)$ is self-dual up to a twist which is \emph{constant} over $\II$.

  It is helpful to note that $V_p(\ucG)^*$ has a canonical realisation in the cohomology of a Shimura variety attached to $\GL_2 \times \GL_2 \times \mathbf{G}_m$. We can regard the scheme $\mu_{p^r}^\circ$ of primitive $p^r$-th roots of unity, with its natural Galois action, as a Shimura variety for $\mathbf{G}_m$. We can then realise $V_p(\ucG)^*$ as a quotient of the module
  \[
   (e'_{\ord}, e'_{\ord})\varprojlim_r  H^2_{\et}\left(
   \left(Y_1(Np^r) \times Y_1(Np^r) \times \mu_{p^r}^\circ\right)_{\QQbar}
   , \Zp(1)\right) \htimes \II;
  \]
  it is the maximal quotient on which $\left(\stbt{\star}{\star}{}x,\stbt{\star}{\star}{}y, z\right)$ for $x,y,z \in \Zp^\times$ acts as multiplication by the character $x^{(\mathbf{c}_1 - 2)} y^{(\mathbf{c}_2 - 2)} z^{(2-\mathbf{c})}$, and the transpose Hecke operators $T'_\ell$ away from $p$ on the two $\GL_2$ factors act by the systems of eigenvalues associated to the $\cG_i$. In particular, elements of the form $\left(\stbt{x}{\star}{}{x},\stbt{x}{\star}{}{x}, x^2\right)$ act trivially if $\chi_1 \chi_2$ is even, and and as multiplication by $x$ if $\chi_1\chi_2$ is odd.

 \subsection{Galois representations for $\GSp_4$}
  \label{sect:gsp4gal}

  Let $\Pi$ be a cuspidal automorphic representation of $\GSp_4$ such that:
  \begin{enumerate}[(i)]
   \item $\Pi$ is globally generic, and is not CAP or a Yoshida lift.
   \item $\Pi_\infty$ is the generic discrete series representation $\Pi_\infty^{\mathrm{W}}$ of weight $(k_1, k_2)$, for some integers $k_1 \ge k_2 \ge 3$.
  \end{enumerate}

  Write $(r_1, r_2) = (k_1 - 3, k_2 - 3)$; thus $r_1 \ge r_2 \ge 0$, so $\lambda = (r_1, r_2; r_1 + r_2)$ is the highest weight of an irreducible algebraic representation of $G$. As in \cite{LPSZ1}, we define a Galois representation $V_p(\Pi)$ associated to $\Pi$; this Galois representation is 4-dimensional, and its $L$-function is $L(\Pi, s - \frac{r_1 + r_2 + 3}{2})$. It can be realised canonically as the space of new-vectors in the $\Pif$-isotypical part of the \'etale cohomology of the $\GSp_4$ Shimura variety over $\QQbar$, with coefficients in the sheaf $\cV_\lambda$ associated to $\lambda$.

  \subsubsection*{Duality} The Galois representation $V_p(\Pi)$ supports a non-zero, $G_{\QQ}$-equivariant symplectic pairing $V_p(\Pi) \otimes V_p(\Pi) \to L(\chi_{\Pi})(3-k_1-k_2)$, where $\chi_{\Pi}$ is regarded as a Galois character via the Artin map (normalised to send uniformisers to geometric Frobenii).

  \subsubsection*{Ordinarity} If $\Pi_p$ is ordinary, then $V_p(\Pi)$ has a complete flag of subspaces $\Fil^i V_p(\Pi)$ stable under $\Gal(\QQbar_p / \Qp)$, with $\Fil^i$ of codimension $i$, and $\Fil^{i}$ the annihilator of $\Fil^{4-i}$ with respect to the symplectic pairing. If $\Pi_p$ is Klingen-ordinary, then $\Fil^2$ exists (but $\Fil^1$ and $\Fil^3$ may not); on the other hand, if it is Siegel-ordinary, then $\Fil^1$ and $\Fil^3$ exists, but $\Fil^2$ may not.

\section{Conjectures II: Geometric cohomology classes for $-1$ regions}
 \label{sect:conjs2}

 \subsection{Setting}

  We consider the following situation:

  \begin{situation}
  \label{sit:pairGalrep}
   Let $(\Pi, \ucG)$ be as in Situation \ref{sit:pair}; and suppose additionally that:
   \begin{enumerate}[(i)]
    \item $\Pi$ satisfies the conditions of \S \ref{sect:gsp4gal};
    \item $\Pi_p$ is ordinary;
    \item the Hida families $\cG_i$ are residually non-Eisenstein and $p$-distinguished.
   \end{enumerate}
  \end{situation}

  In this section, we turn our attention to the ``sign $-1$'' cases: that is, we shall study specialisations of $(\Pi, \ucG)$ with weights lying in regions $\{b, b', e\}$ when $(\Pi, \ucG)$ is indefinite, or in regions $\{a,a',c,d,d',f\}$ when $(\Pi, \ucG)$ is definite. In this case, the global Gross--Prasad conjecture \ref{conj:GGP} does not apply -- the central $L$-value $\Lambda(\Pi \times \Sigma, \tfrac{1}{2})$ vanishes, and the period $\cP(-)$ is identically 0. The Beilinson--Bloch--Kato conjecture predicts that the vanishing of $\Lambda(\Pi \times \Sigma, \tfrac{1}{2})$ should be related to the existence of an appropriate global Selmer class. We shall formulate a (rather vague) conjecture asserting that these Selmer classes interpolate over the family $\ucG$.

 \subsection{Selmer groups}

  We are interested in the 16-dimensional tensor-product Galois representation
  \[ \VV(\phi) = \Big(V_p(\Pi) \otimes V_p(\Sigma_1(\phi)) \otimes V_p(\Sigma_2(\phi))\Big)\left( 1 + w\right)
  = \Big(V_p(\Pi)^* \otimes V_p(\Sigma_1(\phi))^* \otimes V_p(\Sigma_2(\phi))^*\Big)\left( - w\right),
  \]
  where $w = \frac{r_1 + r_2 + c_1 + c_2}{2}$ as above, for $\phi \in \fX^{\cl}(\ucG)$. We are interested in the following conjecture, which is part of the Beilinson--Bloch--Kato conjecture for this Galois representation:

  \begin{conjecture}
   Let $\spadesuit \in \{ b, b', e\}$ if $\Pi \times \ucG$ is indefinite, or $\spadesuit \notin \{ b, b', e\}$ if $\Pi \times \ucG$ is definite. Then for any $\phi \in \fX^{\spadesuit}$, we should have $H^1_{\f}(\QQ, \VV(\phi)) \ne 0$. Moreover, there should exist a non-zero class in this space which is the image of an algebraic cycle on a product of Kuga--Sato varieties.
  \end{conjecture}

  \begin{proposition}
   \label{prop:h1fg} Let $\phi \in \fX^{\cl}$, with $c_1, c_2 \ge 2$. Then for each prime $\ell$ (including $\ell = p$), we have
   \[ H^1_{\f}(\Ql, \VV(\phi)) = H^1_{\mathrm{g}}(\Ql, \VV(\phi)).\]
  \end{proposition}

  \begin{proof}
   Recall that for a geometric Galois representation $V$ and $\ell \ne p$ we let $H^1_{\mathrm{g}}(\Ql, V) = H^1(\Ql, V)$ and $H^1_{\f}(\Ql, V) = H^1(\Ql^{\mathrm{nr}}/\Ql, V^{I_\ell})$. The definition for $\ell = p$ is more complicated, but in either case we have
   \[
    H^1_{\f}(\Ql, V) = H^1_{\mathrm{g}}(\Ql, V) \qquad\Longleftrightarrow\qquad
    L_\ell(V^*(1), 0) \ne \infty.
   \]
   For $V = \VV(\phi)$, we have $V = V^*(1)$ and $L_\ell(V, 0) = L(\Pi_\ell \times \Sigma_\ell, \tfrac{1}{2})$, and since $\Pi_\ell$ and $\Sigma_\ell$ are tempered, all poles of $L(\Pi_\ell \times \Sigma_\ell, s)$ have real part $\le 0$.
  \end{proof}

  We noted above that for each $\phi \in \fX^{\cl}$, there is a unique 8-dimensional subspace $\VV^+(\phi)\subset \VV(\phi)$, stable under the decomposition group at $p$, such that $\VV^+(\phi)$ has all Hodge--Tate weights $\ge 1$ and the quotient has all weights $\le 0$ (the ``Panchishkin condition'').

  \begin{proposition}
   Suppose $\phi$ is either crystalline, or fully ramified. Then the map $H^1(\Qp, \VV^+(\phi)) \to H^1(\Qp, \VV(\phi))$ is injective, and its image is $H^1_{\f}(\Qp, \VV(\phi))$.
  \end{proposition}

  \begin{proof}
   It suffices to show that neither $1$ nor $p^{-1}$ can be an eigenvalue of the Frobenius on $\Dcris(\Qp,\VV(\phi))$. If $\phi$ is fully ramified, then $\Dcris(\Qp,\VV(\phi)) = 0$ so the statement is clear. If $\phi$ is crystalline, then all eigenvalues of Frobenius on this space have complex absolute value $p^{-1/2}$, so none can be equal to 1 or $p^{-1}$.
  \end{proof}

 \subsection{Self-dual families and $p$-refinements}

  For $(\Pi, \ucG)$ as in \cref{sit:pairGalrep}, we can define a free rank 16 module $\VV_p(\Pi \times \ucG)$ over $\II$ as follows. Let us write $k = \left\lceil \tfrac{k_1 + k_2}{2}\right\rceil$ (integer ceiling function), and define
  \[ \VV_p(\Pi \times \ucG) = V_p(\Pi)(k-1) \otimes V_p(\ucG) = V_p(\Pi)^*(2-k)\otimes V_p(\ucG)^*.\]
  This is a free rank 16 module over $\II$ with a continuous $\II$-linear action of $\Gal(\QQbar / \QQ)$, satisfying $\VV_p(\Pi \times \ucG)^*(1) \cong \VV_p(\Pi \times \ucG)$, which interpolates the $\VV(\phi)$ for $\phi \in \cX^{\cl}(\ucG)$.

  \begin{proposition}
   For each of the 10 regions $\spadesuit$, there exists a rank 8 $\II$-submodule $\VV^+(\Pi \otimes \ucG, \spadesuit) \subset \VV_p(\Pi \times \ucG)$ stable under $\Gal(\QQbar_p / \Qp)$ such that for each $\phi \in \fX^{\spadesuit}$, the specialisation of $\VV^+(\Pi \otimes \ucG, \spadesuit)$ at $\phi$ is the subspace $\VV^+(\phi)$.
  \end{proposition}

  \begin{proof}
   This follows from standard properties of Galois representations associated to automorphic forms for $\GSp_4$ and to Hida families. For instance, in region $(f)$ we use the fact that $V_p(\Pi)$ has a 2-dimensional local subrepresentation at $p$ with Hodge--Tate weights $\{0, -1-r_2\}$. Tensoring this with $V_p(\cG_1) \otimes V_p(\cG_2)$ gives an 8-dimensional subrepresentation of the tensor product; one checks that the specialisations of this subrepresentation at $\phi \in \fX^{\cl}$ satisfies the Panchishkin condition if and only if $\phi \in \fX^{(f)}$.
  \end{proof}

  \begin{remark}
   Compare \cite[Lemma 1.15]{darmonrotger16}. As is clear from the argument given above for region $(f)$, it is not necessary to assume that $\Pi$ and the families $\cG_i$ are nearly-ordinary for the Borel subgroups to construct $\VV^+(\Pi \otimes \ucG, \spadesuit)$; it suffices to assume near-ordinarity for the parabolic subgroup corresponding to $\spadesuit$ in the table of \cref{rem:parabolics}.
  \end{remark}

 \subsection{Families of Selmer groups}

  By standard methods (cf.~\cite{nekovar06}) we can define a Selmer complex $\RGt(\QQ, \Pi \otimes \ucG; \spadesuit)$ in the derived category of $\II$-modules, using $\VV^+(\Pi \otimes \ucG, \spadesuit)$ as the local condition at $p$; the formation of this complex commutes with (derived) base-change, and the $H^1$ of its specialisation at a point $\phi \in \fX^{\spadesuit}$ is exactly $H^1_{\mathrm{f}}(\QQ, \VV(\phi))$.

  \begin{conjecture}
   \label{conj:arithpGGP}
   If $(\Pi, \ucG)$ is indefinite, then
   \[
    \operatorname{rank}_{\II} \widetilde{H}^1(\QQ, \Pi \otimes \ucG; \spadesuit) =
    \begin{cases}
     0 & \text{if $\spadesuit \in \{a,a',c,d,d',f\}$},\\
     1 & \text{if $\spadesuit \in \{b,b',e\}$}.
    \end{cases}
   \]
   Moreover, if $\spadesuit \in \{b,b',e\}$, there should be a non-torsion element $\Delta(\Pi \times \ucG, \spadesuit)$ of this $\widetilde{H}^1$ whose specialisation at every $\phi \in \fX^{\spadesuit}$ is the image under the \'etale cycle class map of a $\QQbar$-linear combination of algebraic cycles. The same should hold for definite $(\Pi, \ucG)$ with the roles of regions $\{a,a',c,d,d',f\}$ and $\{b,b',e\}$ interchanged.
  \end{conjecture}

  This conjecture is considerably less precise than Conjecture \ref{conj:pGGP}, since we have not specified the interpolating property of $\Delta(\Pi \times \ucG, \spadesuit)$ precisely, only up to $\QQbar^\times$. We can remove this ambiguity in either of two ways:
  \begin{itemize}
   \item By specifying precisely which algebraic cycle the class should interpolate. We shall do this below for $\spadesuit = (e)$; we give the details only for $(\Pi,\ucG)$ split, but the construction readily extends to the more general indefinite case. However, in the other regions we do not know how to construct any interesting algebraic cycles.

   \item By specifying the image of $\Delta(\Pi \times \ucG, \spadesuit)$ under the Perrin-Riou regulator map. This will be the topic of our final set of conjectures.
  \end{itemize}

\section{Constructions II: Interpolating cycles in region \texorpdfstring{$(e)$}{(e)}}
 \label{sect:famGale}
 Let $(\Pi, \ucG)$ be as in \cref{sit:pairGalrep}, and assume in addition that $(\Pi, \ucG)$ is split (i.e.~the local root numbers at all finite primes are $+1$).

 We shall now define a collection of cohomology classes in $H^1_{\f}(\QQ, \VV(\phi))$ arising from algebraic cycles, for each specialisation $\phi \in \fX^{(e)}(\ucG)$; subsequently, we shall show that these interpolate as $\phi$ varies.

 \subsection{Geometric classes}

  Let $\phi \in \fX^{(e)}$, and write $\Sigma_i = \Sigma_i(\cG, \phi)$. We shall define, rather than a single cohomology class, a ``cohomology-valued period pairing'' mapping vectors in the Whittaker models of $\Pif$ and $\Sif$ to the cohomology group $H^1(\QQ, \VV)$.

  \subsubsection*{Trivial coefficients} We first give the construction in a simple case: we assume that  $(k_1, k_2) = (3, 3)$ and $(c_1, c_2) = (2, 2)$, so $\Pi$ and $\Sigma$ contribute to \'etale cohomology of the Shimura varieties for $G$ and $H$ with trivial coefficients. Note that this does indeed correspond to a point of region $(e)$.

  \begin{notation} We write $\cK$ for a triple of open compacts $(K_G, K_1, K_2)$ with $K_G \subset G(\Af)$, and $K_1$, $K_2 \subset \GL_2(\Af)$. For such a $\cK$ we write $Y(\cK)$ for the product of Shimura varieties $Y_G(K_G) \times Y_{\GL_2}(K_1) \times Y_{\GL_2}(K_2)$; and we write $\cK \cap H$ for the intersection of $\cK$ with $H(\Af)$, embedded in $\GSp_4 \times \GL_2 \times \GL_2$ via $(h_1, h_2) \mapsto (\iota(h_1,h_2), h_1, h_2)$.
  \end{notation}

  \begin{definition}
   For $\cK$ as above, the \emph{diagonal cycle} of level $\cK$ is the class $\Delta(\cK) \in \operatorname{CH}^3(Y(\cK))$ given by the image of the map of Shimura varieties
   \[ Y_H(\cK\cap H) \xrightarrow{(\iota, \mathrm{id})} Y(\cK).\]
  \end{definition}

  The \'etale cycle class of the diagonal cycle lies in $H^6_{\et}\left(Y(\cK), \Qp(3)\right)$. This space decomposes as a direct sum of generalised eigenspaces for the spherical Hecke algebra; our assumptions on $\Pi$ and the $\Sigma_i$ implies that the $(\Pi^\vee, \Sigma_1^\vee, \Sigma_2^\vee)$ eigenspace does not contribute to cohomology over $\QQbar$ in degrees other than 5, so we obtain a projection map to the space
  \[ H^1\Big(\QQ, H^5_{\et}\left( Y(\cK)_{\QQbar},  \Qp(3)\right)[\Pi^\vee \times \Sigma_1^\vee \times \Sigma_2^\vee]\Big).\]

  If $\operatorname{vol}_H$ denotes a choice of Haar measure on $H(\Af)$, then one can check that the renormalised classes
  \[
   \operatorname{vol}_H\left(\cK \cap H\right) \cdot \Delta(\cK)
  \]
  are compatible under pushforward maps for varying $\cK$.

  By construction, vectors in the Whittaker model $\cW(\Pif)$ give $\Gal(\QQbar/\QQ)$-equivariant maps from $V_p(\Pi)$ to the $\Pif$-part of the direct limit $\varinjlim_{K_G} H^3_{\et, c}(Y_{G}(K_G)_{\QQbar}, \Qp)$, or dually from $\varprojlim_{K_G} H^3_{\et}(Y_{G, \QQbar}, \cV_\lambda^\vee(3))[\Pif^\vee]$ to $V_p(\Pi)^*$ (``modular parameterisations'' of this Galois representation). The same applies to the $\Sigma_i$; so each triple of Whittaker functions $w_0 \times w_1 \times w_2 \in \cW(\Pif) \times \cW(\Sigma_{1, \f}) \times \cW(\Sigma_{2, \f})$ gives a homomorphism
  \[
   \varprojlim_{\cK} H^5_{\et}\left( Y(\cK)_{\QQbar},  \Qp(3)\right)[\Pi^\vee \times \Sigma_1^\vee \times \Sigma_2^\vee] \to
   V_p(\Pi)^* \otimes V_p(\Sigma_1)^* \otimes V_p(\Sigma_2)^*(-2) = \VV(\phi).\]

  Applying this modular parametrisation to the compatible family of classes $\operatorname{vol}_H\left(\cK \cap H\right) \Delta(\cK)$, we have defined a linear map
  \[ \Delta(\Pi, \Sigma): \cW(\Pif) \otimes \cW(\Sigma_{1, \f}) \otimes \cW(\Sigma_{2, \f}) \to H^1(\QQ, \VV(\phi)); \]
  and one checks readily that this map is $H(\Af)$-equivariant. Using the main theorem of \cite{nekovarniziol} one sees that it takes values in $H^1_{\mathrm{g}}$ locally at $p$; hence it lands in the global $H^1_{\f}$ by Proposition \ref{prop:h1fg}.

 \subsubsection{General coefficients}

  We now relax the assumption on the coefficients, allowing any \emph{cohomological} weights $(k_1, k_2)$ for $\Pi$ and $(c_1, c_2)$ for $\Sigma$. In this case, we shall construct geometric classes using the formalism of \emph{motivic cohomology with coefficients}.

  The coefficient systems arise from highest-weight algebraic representations $V_{G, \lambda}$ and $V_{H, \mu}$. We write $\lambda = (r_1, r_2; r_1 + r_2)$ and $\mu = (t_1, t_2; t_1 + t_2)$ for some $r_i, t_i \ge 0$; these are related to our parameters above by
  \[ (k_1, k_2) = (r_1 + 3, r_2 + 3), \qquad (c_1, c_2) = (t_1 + 2, t_2 + 2).\]
  for some $r_i, t_i \ge 0$. Note that the $H$-representation $V_{G, \lambda} \otimes V_{H, \mu} \otimes \det^{2-w}$ has trivial central character.

  \begin{proposition}
   We have $\Hom_{H}( V_\lambda^G \otimes V_\mu^H \otimes \det^{2-w}, \mathrm{triv}) \ne 0$ if and only if $(c_1, c_2)$ lies in region (e).
  \end{proposition}

  \begin{proof}
   This is a restatement of the standard branching law for algebraic representations.
  \end{proof}

  Associated to these algebraic representations, there are equivariant relative Chow motives $\cV_{\lambda}^G$ and $\cV_\mu^H$ on the Shimura varieties $Y_G$ and $Y_H$, via Ancona's functor (see \cite{ancona15}, \cite[\S 6]{LSZ17}). Using the compatibility of Ancona's functor with morphisms of PEL data established by Torzewski \cite{torzewski18}, for each $\cK$, we have a homomorphism
  \begin{align*} H^0_{\mot}(Y_H(\cK\cap H), \QQ)
  &\to H^0_{\mot}(Y_H(\cK\cap H), \tilde{\iota}^*( \cV_{G, \lambda}^\vee \otimes \cV_{H, \mu}^\vee)(2-w) )\\
  &\to H^6_{\mot}(Y(\cK), \cV_{G, \lambda}^\vee \otimes \cV_{H, \mu}^\vee(5-w) ).
  \end{align*}
  We can define $\Delta^{[r_1, r_2, t_1, t_2]}(\cK)$ to be the image of the identity class $1 \in H^0_{\mot}(Y_H(\cK\cap H), \QQ)$ under this map. As before, these classes become compatible under pushforwards for varying $\cK$ after renormalising by $\operatorname{vol}_H(\cK\cap H)$.

  \begin{remark}
   If $(r_1, r_2) = (t_1, t_2) = (0, 0)$, so that $w = 2$, then the motivic cohomology group $H^6_{\mot}(Y(\cK), \cV_{G, \lambda}^\vee \otimes \cV_{H, \mu}^\vee(5-w) ) = H^6_{\mot}(Y(\cK), \QQ(3) )$ is isomorphic to the Chow group $\operatorname{CH}^3(Y(\cK))_{\QQ}$; and $\Delta^{[r_1, r_2, t_1, t_2]}(\cK)$ is simply the class of the algebraic cycle $\Delta(\cK)$ above. Thus our constructions for general $r_i$ and $t_i$ are indeed a generalisation of the construction given for trivial coefficients above.
  \end{remark}

  Applying the \'etale realisation map and projecting to the $(\Pi^\vee, \Sigma^\vee)$-isotypical part, we obtain a class in the space
  \[
   H^1\Big(\QQ, \varprojlim_{\cK} H^5_{\et}(Y(\cK)_{\QQbar}, \cV_{G, \lambda}^\vee \otimes \cV_{H, \mu}^\vee(5-w) )
   [\Pi^\vee \times \Sigma^\vee]\Big).
  \]
  Just as in the trivial-coefficient case, we can regard vectors in the Whittaker model as $\Gal(\QQbar/\QQ)$-equivariant maps from the above representation to $V(\phi)$. So we obtain a pairing
  \[ \Delta(\Pi, \Sigma): \cW(\Pif) \otimes \cW(\Sigma_{1, \f}) \otimes \cW(\Sigma_{2, \f}) \to H^1_{\f}(\QQ, \VV(\phi)) \]
  as before.

  \begin{remark}
   The \emph{arithmetic Gross--Prasad conjecture}\footnote{We have not been able to find a precise statement of such a conjecture in the literature, and we shall not attempt to formulate one here. However, the formulation of the analogous conjecture for unitary groups is standard; and the expectation that it should generalise to orthogonal, and spin similitude, groups is well-known to experts.} for $\GSpin(5) \times \GSpin(4)$ would posit a relation between the Beilinson--Bloch height of the motivic class $\Delta^{[r_1, r_2, t_1, t_2]}(\cK)$ (projected to the $(\Pi, \Sigma)$-isotypical component) and the value $\Lambda'(\Pi \times \Sigma, \tfrac{1}{2})$. We shall not consider height pairings or derivatives of complex $L$-values explicitly in the present work, but this conjecture serves as important motivation for expecting these classes to be arithmetically interesting.
  \end{remark}

 \subsection{Interpolation in families}

  \subsubsection{Local data at $p$}

   Let $\phi \in \fX^{(e)}$. We want to define vectors in the local Whittaker models $\cW(\Pi_p)$, $\cW(\Sigma_{1,p}(\phi))$, and $\cW(\Sigma_{2,p}(\phi))$ which are suited to $p$-adic interpolation; note that the recipe here is somewhat different from region $(f)$ above (and will give a different Euler factor in the crystalline case).

   Let $\Iw(p) = \{ g \in G(\Zp): g\bmod p \in B_G\}$. The space $\cW(\Pi_p)^{\Iw(p)}$ is 8-dimensional, and there is a 1-dimensional eigenspace on which the (integrally normalised) Hecke operators $U_{1, \Iw}$ and $U_{2, \Iw}$ act as $p$-adic units, with eigenvalues $\alpha$ and $\tfrac{\alpha\beta}{p^{k_2 - 2}}$ respectively. This eigenspace has a unique basis vector $w_{\alpha\beta}^{\Iw}$ such that $w_{\alpha\beta}^{\Iw}(1) = 1$.

   \begin{lemma}
    The lower Borel $\overline{\Bor}_H$ of $H$ acts on the flag variety $G / \Bor_G$ with a unique open orbit, represented by the element
    \[ u_\Bor =
     \begin{smatrix}
      1 \\
      1 & \phantom{-}1 \\
      0 & \phantom{-}1 & \phantom{-}1\\
      0 & -1 &-1& 1
     \end{smatrix} \in \overline{N}_G(\Zp).
    \]
    The $\overline{B}_H$-stabiliser of any point in this orbit is equal to $H \cap Z_G$.
   \end{lemma}

   (Other coset representatives are possible, of course; we have chosen one which lifts the element $u_{\Kl}$ considered in \cref{prop:region-f-zeta}.)

   \begin{lemma}
    \label{lem:zeta-e-limit}
    Let $w_i \in \cW(\Sigma_{i, p})$ be vectors such that $w_1 \otimes w_2$ is stable under the subgroup of $H(\Zp)$ consisting of elements of the form $\left(\stbt{z}{\star}{}{z},\stbt{z}{\star}{}{z}\right)$, and the Hecke operator $U_p$ acts on $w_i$ with eigenvalue $\alpha_i$.

    Then, for any $\mathfrak{z}_p \in \Hom(\cW(\Pi_p) \otimes \cW(\Sigma_{1, p}) \otimes \cW(\Sigma_{2,p}), \CC)$, the values
    \[
     \left(\frac{p^{(3k_1 + 3k_2 + c_1 + c_2 - 2)/2}}{\alpha^2\beta \alpha_1 \alpha_2}\right)^n \mathfrak{z}_p\Bigg(
      u_\Bor J \diag(p^{3n}, p^{2n}, p^n, 1)  w_{\alpha\beta}^{\Iw},
      \stbt{p^n}{}{}{1} w_1,
      \stbt{p^n}{}{}{1} w_2
     \Bigg)
    \]
    stabilise for $n \gg 0$.
   \end{lemma}

   It suffices to prove the lemma assuming $k_1 = k_2 = 3, c_1 = c_2 = 2$ (since the $k_i$ and $c_i$ appear only in the normalisation of the $U$-operators). In this case, the normalisation factor is $\left(\frac{p^{10}}{\alpha^2\beta\alpha_1\alpha_2}\right)^n$.

   We note that the elements $\left(\frac{p^{8}}{\alpha^2 \beta}\right)^n J \diag(p^{3n}, p^{2n}, p^n, 1)  w_{\alpha\beta}^{\Iw}$ are invariant under a sequence of upper-triangular congruence subgroups, whose intersection is $B_G(\Zp)$; these vectors are compatible under the normalised trace maps, and all are eigenvectors for $U'_1$ and $U'_2$ with eigenvalues $\alpha$ and $\tfrac{\alpha\beta}{p}$ respectively.

   \begin{proposition}
    Suppose $\Sigma_{1, p}$ and $\Sigma_{2, p}$ are unramified, and let $w_i$ be the normalised $\alpha_i$-eigenvectors in the Whittaker model at $\GL_2$ Iwahori level (so $w_i(1) = 1$).

    If $\mathfrak{z}_p$ is the homomorphism with $\mathfrak{z}_p(w_0^{\sph}, w_1^{\sph}, w_2^{\sph}) = 1$, then the sequence of Lemma \ref{lem:zeta-e-limit} is independent of $n \ge 1$, and the common value is given by
    \[ \frac{q^4}{(q^2 - 1)^2} \mathcal{E}_p^{(e)}(\Pi_p \times \Sigma_p(\phi)).\]
    In particular, it is not zero.
   \end{proposition}

   \begin{proof}
    See \cite{loeffler-zeta2}.
   \end{proof}

   For ramified points $\phi$, we must be a little more circumspect. Recall that $\Sigma_{1,p}$ is ordinary, but $\Sigma_{2, p}$ is a twist of an ordinary representation $\Sigma_p(\cG_2, \phi_2)$ by a possibly ramified character $\widehat{\tau}_{\phi, p}^{-1}$ (which coincides with $\tau_{\phi}$ on $\Zp^\times$, and maps $p$ to $1$.)

   We shall define $w_1$ to be the normalised $(U = \alpha_1)$-eigenvector in the $\stbt{\star}{\star}{}1$-invariants of $\Sigma_{1,p}$, and $w_2$ to be the twist by $\widehat{\tau}_{\phi, p}^{-1}$ of the $(U = \alpha_2)$-eigenvector of $\Sigma_p(\cG_2, \phi_2)$ (which is \emph{not} the same as the new vector of the twist $\Sigma_{2,p}$, in contrast to \S\ref{sect:pLf} above). Thus $w_1 \otimes w_2$ is invariant under the subgroup $\left(\stbt{z}{\star}{}{z},\stbt{z}{\star}{}{z}\right)$ of $N_H(\Zp)$, as required for Lemma \ref{lem:zeta-e-limit}, and the normalised $U$-operator $p^{(c_i-2)/2} \sum_{a \bmod p} \stbt{p}{a}{}{1}$ acts on $w_{i}$ as the $p$-adic unit $\alpha_i$. (These conditions in fact determine $w_1 \otimes w_2$ uniquely up to scalars.)

   \begin{remark}
    It should be possible to show that if $\Hom(\cW(\Pi_p) \otimes \cW(\Sigma_{1, p}) \otimes \cW(\Sigma_{2,p}), \CC)$ is non-zero, then the limit of \cref{lem:zeta-e-limit} is non-zero for these vectors $w_i$ (at least if $\phi$ is fully ramified, as in the analogous computation above for region $(f)$). However, we have not attempted to do this.
   \end{remark}

  \subsubsection{Interpolation}

   We now define elements $w_0 \in  \cW(\Pif)$, $w_{i, \phi} \in \cW(\Sigma_{i,\f})$ as the product of the vectors of \cref{lem:zeta-e-limit} at $p$ (for some sufficiently large $n$) and the vectors $\gamma_{i, v}\cdot  w_{i, v}^{\new}$ away from $p$, where $\gamma = (\gamma_0, \gamma_1, \gamma_2)$ is our fixed, arbitrary element of $(G \times \widetilde{H})(\Af^p)$.

  \begin{theorem}
   There exists a class
   \[ \Delta_{\gamma}(\Pi \times \ucG, (e)) \in \wH^1(\QQ, \VV(\Pi \times \ucG); (e)) \]
   with the following property: for every $\phi \in \fX^{(e)}$, we have
   \[ \Delta_{\gamma}(\Pi \times \ucG, (e))(\phi) = \Delta(\Pi, \Sigma)(w_0, w_{1, \phi}, w_{2, \phi}) \]
   where $w_0$ and $w_{i, \phi}$ are the above Whittaker vectors. If $\phi$ is crystalline, we have
   \[ \Delta(\Pi, \Sigma)(w_0, w_{1, \phi}, w_{2, \phi}) = \frac{q^4}{(q^2 - 1)^2} \mathcal{E}_p(\Pi \times \Sigma(\phi))  \times \Delta(\Pi, \Sigma)(w_0^{\sph}, w_{1, \phi}^{\sph}, w_{2, \phi}^{\sph}).\]
  \end{theorem}

  The proof of this theorem is an instance of more general results proved in \cite{loeffler-spherical} and its forthcoming sequel paper \cite{LRZ}. We give a brief sketch here. Let us fix a prime-to-$p$ level $\cK$. Using the methods of \cite{loeffler-spherical}, for any given pair $(t_1, t_2)$ satisfying the required inequalities, one can construct an element of $H^1(\QQ, \VV(\Pi \otimes \cG); (e))$ interpolating the classes $\Delta^{[r_1, r_2, t_1, t_2]}(\cK \cap V_r)$ for varying $r$, where $V_r$ is a certain family of level groups at $p$. This construction relies on the fact that the group $H$ acts on the flag variety $(G \times \tilde H) / (\overline{\Bor}_G \times \overline{\Bor}_{\tilde{H}})$ with an open orbit, and the family of interpolating elements depends on a choice of representative of this orbit; we use the element $(u_{\Bor} J, \mathrm{id})$.

  More difficult is the fact that there is a \emph{single} element interpolating the $\Delta^{[r_1, r_2, t_1, t_2]}(\cK \cap V_r)$ for all $(t_1, t_2)$. (In fact we can also vary $(r_1, r_2)$ as well, although we shall not use this fact here.) This follows from a Lie-algebra-theoretic computation analogous to \cite[Lemma 4.4.1 \& Theorem 9.6.4]{LSZ17}; the details will be given in \cite{LRZ}.

\section{Conjectures III: Explicit reciprocity laws}

 We now choose a pair $(\Pi, \ucG)$ satisfying our usual assumptions, and \emph{two} regions $\spadesuit$, $\heartsuit$, which we suppose to be adjacent (sharing an edge) in \cref{fig:GGP}. This implies that the global root number is $+1$ in one of the regions and $-1$ in the other. Our final conjecture will be an explicit reciprocity law, relating the $p$-adic $L$-function for the $+1$ region to the Galois cohomology class for the $-1$ region.

 \subsection{Ordinary filtrations}

  We first analyse the Galois representation $V_p(\Pi)$. Recall that this has non-negative Hodge numbers $\{0, \dots, k_1 + k_2 - 3\}$. Our conjectures will be vacuous unless $k_2 \ge 3$, so we shall assume this; hence the Hodge numbers of $\Pi$ are distinct. Let these be $n_0 = 0, \dots, n_3 = k_1 + k_2 - 3$.

  We shall also suppose that $\Pi$ is ordinary at $p$. Then $V_p(\Pi)$ has a filtration by subrepresentations $V_p(\Pi) = \Fil^0 V_p(\Pi) \supset \Fil^1 V_p(\Pi) \dots$, with graded pieces $\Gr^i = \Fil^i / \Fil^{i + 1}$ of dimension 1 for $i\in \{0,1,2,3\}$. We refer to the ensuing filtration on $\DdR(V_p(\Pi))$ as the \emph{Newton filtration}.

  The Newton filtration is ``opposite'' to the Hodge filtration on $\DdR(V_p(\Pi))$, in the sense that $\Fil^{i}_{\cN} \cap \Fil^j_{\cH}$ has dimension 1 if $j = n_{3-i}$ and is zero if $j > n_{3-i}$. Hence we have isomorphisms
  \[ \Gr^{n_{(3-i)}}_{\cH} \cong \Fil^{i}_{\cN} \cap \Fil^{n_{(3-i)}}_{\cH} \cong \Gr^{i}_{\cN}.\]

  However, $\Gr^{n_{(3-i)}}_{\cH} \DdR(V_p(\Pi))$ is isomorphic to the newvectors of the $\Pi$-part of coherent $H^i$ of the compactified Shimura variety, with a suitable coefficient system; we call this space $H^i_{\mathrm{coh}}(\Pi)^{\new}$. So we conclude that there is an isomorphism $\DdR(\Gr^i V_p(\Pi)) \cong H^i_{\mathrm{coh}}(\Pi)^{\new}$, or dually that there is a canonical pairing
  \[
   \DdR(\Gr^i V_p(\Pi^\vee)) \times H^{3-i}_{\mathrm{coh}}(\Pi)^{\new} \to \QQbar_p.
  \]

  Similarly, we have 2-step filtrations on the 2-dimensional representations $V_p(\Sigma_i(\phi))$ for each $\phi \in \fX^{\cl}$, and we obtain pairings
  \[
   \DdR(\Gr^j V_p(\Sigma_i(\phi)^\vee)) \times H^{1-j}_{\mathrm{coh}}(\Pi)^{\new} \to \QQbar_p.
  \]

 \subsection{Local subquotients at $p$ and regulators}

  \begin{proposition}
   The quotient
   \[ \VV(\heartsuit / \spadesuit) \coloneqq \frac{\VV^+(\Pi \times \ucG, \heartsuit)}{\VV^+(\Pi \times \ucG, \heartsuit) \cap \VV^+(\Pi \times \ucG, \spadesuit)}
   \]
   has rank 1. Moreover, the classical specialisations of this quotient have Hodge--Tate weight $\ge 1$ in region $\fX^{\heartsuit}$, and $\le 0$ in region $\fX^{\spadesuit}$. Up to twists, we have
   \[ \VV(\heartsuit / \spadesuit) = \Gr^{s_0} V_p(\pi) \otimes \Gr^{s_1} V_p(\cG_1) \otimes \Gr^{s_2} V_p(\cG_2) \]
   for some integers $(s_0, s_1, s_2) \in \{0,\dots,3\} \times \{0,1\} \times \{0, 1\}$.
  \end{proposition}

  \begin{proof}
   This is a straightforward combinatorial exercise. Since $\VV$ is the product of three representations $\VV_0 \otimes \VV_1 \otimes \VV_2$, each of which has a complete flag of $\Gal(\QQbar_p / \Qp)$-invariant submodules, we obtain a complete description of the Jordan--H\"older constituents of $\VV$ as tensor products of graded pieces of the individual factors. The submodules $\VV^+(\Pi \times \ucG, \spadesuit)$ and $\VV^+(\Pi \times \ucG, \heartsuit)$ correspond to different subsets of these graded pieces.
  \end{proof}

  \begin{example}
   The crucial examples for what follows are:
   \begin{itemize}
    \item $\heartsuit = (e)$, $\spadesuit = (f)$. One computes that $\VV(e/f) = \operatorname{Gr}^{1} \otimes \operatorname{Gr}^{1} \otimes \operatorname{Gr}^{1}$. Hence we should expect to pair with a class in coherent $H^2 \otimes H^0 \otimes H^0$.
    \item $\heartsuit = (e)$, $\spadesuit = (c)$. We have $\VV(e/c) = \operatorname{Gr}^{3} \otimes \operatorname{Gr}^{0} \otimes \operatorname{Gr}^{0}$.
    \end{itemize}
     in the former case, and with a class in coherent $H^0 \otimes H^1 \otimes H^1$ in the latter.
  \end{example}

  \begin{remark}
   Note that interchanging $\spadesuit$ and $\heartsuit$ corresponds to replacing $s = (s_0, s_1, s_2)$ with $s' = (3-s_0, 1-s_1, 1-s_2)$.
  \end{remark}

  We define
  \[ \mathbf{D}(\heartsuit / \spadesuit) \coloneqq \left(\VV(\heartsuit / \spadesuit) \mathop{\hat\otimes} \widehat{\QQ}_p^{\mathrm{nr}}\right)^{\Gal(\QQbar_p / \Qp(\mu_{p^\infty}))}.\]
  This is (non-canonically) isomorphic to $\VV(\heartsuit/\spadesuit)$ as an $\II$-module, so it is free of rank 1. (Alternatively -- less canonically, but slightly more concretely -- we can define it as $\Dcris(\Qp, \VV(\heartsuit / \spadesuit)(\kappa^{-1}))$, where $\kappa$ is the unique $\II^\times$-valued character of $\Gal(\Qp(\mu_{p^\infty}) / \Qp) \cong \Zp^\times$ such that $\VV(\heartsuit / \spadesuit)(\kappa^{-1})$ is unramified.) The theory of Perrin-Riou regulator maps gives a canonical homomorphism of $\II$-modules
  \[
  \cL^{\mathrm{PR}}_{\heartsuit/\spadesuit}: H^1(\Qp, \VV(\heartsuit / \spadesuit)) \to \mathbf{D}(\heartsuit / \spadesuit), \]
  which is related to the Bloch--Kato logarithm for classical specialisations where the Hodge--Tate weight is $\ge 1$ (including those in $\fX^{\heartsuit}$), and to the dual exponential where the Hodge--Tate weight is $\le 0$ (including $\fX^{\spadesuit}$).

 \subsection{Eichler--Shimura isomorphisms}

  For each $\phi \in \fX^{\cl}$, the fibre $\mathbf{D}(\heartsuit / \spadesuit)_{\phi}$ is equal to \[ \Gr^s_{\cH} \DdR(\VV(\phi)) \cong \left(\Gr^{s_0}_{\cH} \DdR(V_p(\Pi)\right) \otimes \left(\Gr^{s_1}_{\cH} \DdR(V_p(\Sigma_1(\phi))\right) \otimes \left(\Gr^{s_2}_{\cH} \DdR(V_p(\Sigma_2(\phi))\right)\]
  where $s = (s_0, s_1, s_2)$ is the triple of integers associated to $\heartsuit$ and $\spadesuit$ as above. As we have seen, this is dual to $H^{s_0'}(\Pi)^{\new} \otimes H^{s_1'}(\Sigma_1(\phi)) \otimes H^{s_2'}(\Sigma_2(\phi))^{\new}$, where $s' = (3, 1,1) - s$.

  The spaces $H^{0}(\Sigma_i(\phi))$ both have canonical bases, given by the newforms $g_i(\phi)$. Using Serre duality, we also obtain bases of the corresponding $H^1$ using the linear functionals dual to the conjugate forms $\overline{g_i(\phi)}$; see \cite{KLZ20} for further details.

  If we fix an arbitrary basis $\nu$ of $H^{s_0'}(\Pi)^{\new}$ (as we already did for $s_0' = 2$ above), then we obtain a canonical ``evaluation'' map
  \[ \operatorname{ev}_{\phi}: \mathbf{D}(\heartsuit / \spadesuit)_{\phi} \to \QQbar. \]

  \begin{proposition}
   There exists a map
   \[ \operatorname{EV}: \mathbf{D}(\heartsuit / \spadesuit) \to \operatorname{Frac} \II \]
   interpolating the $ev_{\phi}$ for $\phi \in \fX^{\cl}$.
  \end{proposition}

 \subsection{The conjecture}

  We shall now assume, without loss of generality, that the global root number is $-1$ in region $\heartsuit$. Hence it is $+1$ in region $\spadesuit$.

  By the construction of the Selmer complex appearing in \cref{conj:arithpGGP}, the localisation map
  \[ \wH^1(\QQ, \VV; \heartsuit) \to H^1(\Qp, \VV) \]
  factors canonically through $H^1(\Qp, \VV^+(\heartsuit))$. We can therefore consider the composite
  \begin{equation}
   \label{evmaps}
   \begin{aligned}
   \wH^1(\QQ, \VV; \heartsuit) &\to H^1(\Qp, \VV^+(\heartsuit)) \\
   &\to H^1(\Qp, \VV(\heartsuit / \spadesuit))\\
   &\xrightarrow{\ \cL^{\mathrm{PR}}_{\heartsuit/\spadesuit}\ }\mathbf{D}(\heartsuit / \spadesuit) \\
   &\xrightarrow{\operatorname{EV}} \operatorname{Frac} \II.
  \end{aligned}
  \end{equation}

  \begin{conjecture}[Explicit reciprocity law]
  \label{conj:explrecip}
   The chain of maps \eqref{evmaps} sends the cohomology class $\Delta(\Pi \times \ucG, \heartsuit)$ of Conjecture \ref{conj:arithpGGP} to the $p$-adic $L$-function $\cL_p(\Pi \times \ucG, \spadesuit)$ of Conjecture \ref{conj:pGGP}.
  \end{conjecture}

 \section{Constructions III: The explicit reciprocity law for regions $(e)$, $(f)$}

  We now place ourselves in \cref{sit:pairGalrep}, and we suppose also that $(\Pi, \ucG)$ is split and that $k_2 \ge 4$, so the hypotheses of both \cref{sect:pLf} and \cref{sect:famGale} are satisfied.

  The constructions above give us candidates for both $\Delta(\Pi \times \ucG, (e))$ and $\cL_p(\Pi \times \ucG, (f))$, depending on certain additional data:
  \begin{itemize}
  \item both depend on a choice of $\gamma_S \in (\GSp_4 \times \GL_2 \times \GL_2)(\QQ_S)$, where $S$ is the set of bad primes $\ell \ne p$;
  \item $\cL_p(\Pi \times \ucG, \spadesuit)$ also depends on the choice of a basis $\nu$ of $H^2(\Pi)^{\new}$.
  \end{itemize}

  \begin{theorem}
   Suppose that the same $\gamma_S$ is used in both constructions, and the $\nu$ in the definition of the $p$-adic $L$-function agrees with that used to define $\operatorname{EV}$. Then the two sides of \cref{conj:explrecip} for $\heartsuit = (e)$, $\spadesuit = (f)$ agree at all points $\phi \in \fX_{e}^{\cris}$.
  \end{theorem}

  The proof of this theorem will be given in a separate paper \cite{LZ20b-regulator}, since it overlaps substantially with the proof of an analogous result in which one of the $\cG_i$ is replaced by an Eisenstein series, which we have not considered here.

  Note that $\fX_{e}^{\cris}$ is a finite set, and hence is far from being Zariski-dense in $\fX$; thus the above result falls short of a proof of \cref{conj:explrecip}. However, in the setting of \cite{LZ20} (where \emph{both} $\cG_i$ are replaced by Eisenstein series), we used deformation in a larger-dimensional family, with the weights of $\Pi$ varying as well, in order to pass from the above theorem to a full proof of the explicit reciprocity law. This strategy relies on the existence of multi-variable $p$-adic $L$-functions for symplectic-type representations of $\GL(4)$ (lifted from $\GSp(4)$), via Shalika models. It does not seem to be possible to extend this method to the case of cuspidal $\cG_i$. However, in future work, we hope to give a direct proof of the existence of these $p$-adic $L$-functions in families for $\GSp(4) \times \GL(2)$ and $\GSp(4) \times \GL(2) \times \GL(2)$ using higher Hida and Coleman theory, which will permit a full proof of \cref{conj:explrecip} in this case.

\newcommand{\noopsort}[1]{\relax}
\newcommand{\etalchar}[1]{$^{#1}$}
\providecommand{\bysame}{\leavevmode\hbox to3em{\hrulefill}\thinspace}
\providecommand{\MR}[1]{}
\renewcommand{\MR}[1]{%
 MR \href{http://www.ams.org/mathscinet-getitem?mr=#1}{#1}.
}
\providecommand{\href}[2]{#2}
\newcommand{\articlehref}[2]{\href{#1}{#2}}

\end{document}